\documentclass[11pt,a4paper]{scrartcl}
\usepackage{amsmath}
\usepackage{amssymb}
\usepackage{amsthm}
\usepackage[english]{babel}
\usepackage{verbatim}
\usepackage[shortlabels]{enumitem}
\usepackage{tikz-cd}
\usepackage{etoolbox} 
\usepackage[all]{xy}
\usepackage{color}

\newtheoremstyle{mio}%
	{}{} 
	{\itshape}{} 
	{\bfseries}{.}{ } 
	{#1 #2\thmnote{~\mdseries(#3)}} 
\theoremstyle{mio}
\newtheorem{teor}{Theorem}[section]

\newtheorem{prop}[teor]{Proposition}
\newtheorem{lemma}[teor]{Lemma}

\newtheorem{Prop}[teor]{Proposition}

\theoremstyle{definition}
\newtheorem{Def}[teor]{Definition}
\newtheorem{Ex}[teor]{Example}

\newtheorem{oss}[teor]{Remark}

\newcommand{\Br}{\mathrm{Br}}

\newcommand{\degdom}{\mathrm{degdom}}

\newcommand{\ins}[1]{\mathbb{#1}}
\newcommand{\insN}{\ins{N}}
\newcommand{\insZ}{\ins{Z}}
\newcommand{\insQ}{\ins{Q}}
\newcommand{\insR}{\ins{R}}
\newcommand{\inN}{\in\insN}
\newcommand{\inZ}{\in\insZ}

\newcommand{\Z}{\mathbb{Z}}

\newcommand{\Q}{\mathbb{Q}}

\newcommand{\Gal}{\mathrm{Gal}}

\newcommand{\inslim}{\mathcal{L}}

\newcommand{\Int}{\textrm{Int}}

\title{Extending valuations to the field of rational functions using pseudo-monotone sequences}

\author{
G.\ Peruginelli\footnote{Dipartimento di Matematica "Tullio Levi-Civita", University of Padova, Via Trieste, 63
35121 Padova, Italy. E-mail: gperugin@math.unipd.it}\\
D. Spirito\footnote{Dipartimento di Matematica e Fisica, University of Roma Tre, Largo San Leonardo Murialdo 1, 00146 Roma. E-mail: spirito@mat.uniroma3.it}
}

\begin{document}
\maketitle

\begin{abstract}
Let $V$ be a valuation domain with quotient field $K$. We show how to describe all extensions of $V$ to $K(X)$ when the $V$-adic completion $\widehat{K}$ is algebraically closed, generalizing a similar result obtained by Ostrowski in the case of one-dimensional valuation domains. This is accomplished by realizing such extensions by means of pseudo-monotone sequences, a generalization of pseudo-convergent sequences introduced by Chabert. We also show that the valuation rings associated to pseudo-convergent and pseudo-divergent sequences (two classes of pseudo-monotone sequences) roughly correspond, respectively, to the closed and the open balls of $K$ in the topology induced by $V$.\\


\noindent MSC Primary 12J20, 13A18, 13F30.
\end{abstract}
\section{Introduction}

Throughout the paper, $V$ will denote a valuation domain with quotient field $K$ and maximal ideal $M$, $v$ will denote its valuation and $\Gamma_v$ its value group. We also fix an algebraic closure $\overline{K}$ of $K$. The study of extensions of $V$ is one of the central parts of valuation theory, which naturally splits into the study of algebraic and purely transcendental extensions. The former can be considered a generalization of the fundamental problems of algebraic number theory, and is well-studied through the concepts of inertia, decomposition and ramification (in what is known as \emph{ramification theory}). The latter -- which is essentially the study of extensions of $V$ to function fields -- is less well understood, but plays a main role in several facets and applications of the theory (see \cite{kuhlmann-transactions} and the references therein). The first step of this problem is to classify all the extensions of $V$ to the rational function field $K(X)$.

In case $V$ has rank one, there are two classical approaches to this problem: the most famous one, due to MacLane, uses key polynomials and augmented valuations and works for arbitrary fields $K$, but requires the valuation ring to be discrete \cite{MacLane}; it has been recently  generalized by Vaqui\`e in \cite{Vaquie} for general valuation domains. The second approach, due to Ostrowski, ``makes no discreteness assumptions'' but ``requires an elaborate construction to obtain values of $\overline K$ from those of $K$'', as MacLane acknowledged in his paper \cite[p. 380]{MacLane}. More precisely, Ostrowski showed that, for a given extension $W$ of $V$ to $K(X)$, there exists a \emph{pseudo-convergent sequence} $E=\{s_n\}_{n\inN}\subset \overline K$ with respect to a extension $u$ of $v$ to $\overline K$ (we refer to \S \ref{pmonotone sequences} for the definition) such that the valuation $w$ associated to $W$ is given  by the real limit $w(\phi)=\lim_{n\to\infty}u(\phi(s_n))$, for all $\phi\in K(X)$; for its importance, Ostrowski called this result \emph{Fundamentalsatz} \cite[\S 11, IX, p. 378]{Ostrowski}.
To our knowledge, Ostrowski's Fundamentalsatz seems to have been mostly forgotten (except in the survey \cite{Roquette}), even if pseudo-convergent sequences have enjoyed some success: for example, Kaplansky used them to characterize immediate extensions of a valued field and maximal fields in \cite{kaplansky-maxfield}, and they are linked to the recently introduced notion of approximation type (see  \cite{approxtype}).  

In generalizing Ostrowski's Fundamentalsatz, we realized that when dealing with the general case (i.e., when the rank of $V$ or of the extension of $V$ to $K(X)$ is not one), pseudo-convergent sequences are not enough to construct all extensions of $V$ to $K(X)$ (see Example \ref{Vinfty}): for this reason, we use the more general notion of \emph{pseudo-monotone sequences}, used in \cite{PerPrufer} to  encompass Ostrowski's notion of pseudo-convergent sequence and the two other kinds of sequences introduced by Chabert in 2010 (namely pseudo-divergent and pseudo-stationary sequences) in order to characterize the so-called \emph{polynomial closure} in the context of rings of integer-valued polynomials. We recall that, given a subset $S$ of $V$, the ring of \emph{integer-valued polynomials} over $S$ is classically defined as $\Int(S,V)=\{f\in K[X] \mid f(S)\subseteq V\}$, and the polynomial closure  of $S$ is the largest subset $\overline{S}\subseteq V$ such that $\Int(S,V)=\Int(\overline{S},V)$. One of the main results of Chabert was to prove that, when $V$ has rank one, the polynomial closure is the closure operator associated to a topology on $K$ (extending the case when $V$ is discrete, originally proved by McQuillan in \cite[Lemma 2]{McQ}). Chabert obtained his result by describing $\overline{S}$ through the set of pseudo-limits of the pseudo-monotone sequences contained in $S$.

In this paper, continuing our earlier work in \cite{PS}, we describe the extensions of $V$ to $K(X)$ by means of pseudo-monotone sequences of $K$, generalizing a natural construction of Loper and Werner, who were interested in studying when the ring of integer-valued polynomials over a pseudo-convergent sequence is a Pr\"ufer domain \cite{LopWer}. More precisely, we associate to every pseudo-monotone sequence $E=\{s_{\nu}\}_{\nu\in\Lambda}\subset K$ (see \S \ref{pmonotone sequences} for the definition) the valuation domain
\begin{equation*}
V_E=\{\phi\in K(X) \mid \phi(s_\nu)\in V, \textnormal{ for all sufficiently large }\nu\in\Lambda\}.
\end{equation*}
We first study the properties of $V_E$ in relation with the properties of $E$; subsequently, we analyze when and how it is possible to associate to an arbitrary extension a pseudo-monotone sequence. Our main result (Theorem \ref{teor:fundamentalsatz}) proves  that every extension  of $V$ to $K(X)$ can be realized in this way if and only if the $v$-adic completion $\widehat{K}$ of $K$ is algebraically closed. In particular, the statement holds if $K$ is algebraically closed, giving a generalization of Ostrowski's result. We also show that, under the same condition, every extension of $V$ to $K(X)$ which is not immediate is a monomial valuation, a natural way of constructing extensions to the field of rational functions (see \S \ref{monomial}).  

\medskip

The structure of the paper is as follows. In Section \ref{notation}, after settling the notation used throughout the paper, and the notions of monomial valuation and divisorial ideal, we give the definition of pseudo-monotone sequence in a general valued field $(K,v)$; we note that Chabert's original definitions of pseudo-divergent and pseudo-stationary sequences  were  given only for a rank one valuation, but they easily extend to the general case. We then introduce the notions of pseudo-limit, breadth ideal and gauge separately for the three different types of pseudo-monotone sequences: pseudo-convergent sequences (\S \ref{pcv}), pseudo-divergent sequences (\S \ref{pdv}) and pseudo-stationary sequences (\S \ref{pst}). In the last part of that section, we characterize pseudo-limits and breadth ideals of pseudo-monotone sequences according to their type (Lemmas \ref{pseudolimiti} and \ref{lemma:carattbreadth}).

In Section \ref{valuations associated pmt} we show that the sequence of values  of the images under a rational function of a pseudo-monotone sequence is eventually monotone (Proposition \ref{values rat func on psm seq}); the result is accomplished by introducing the notion of \emph{dominating degree} of a rational function $\phi\in K(X)$ with respect to a pseudo-monotone sequence $E\subset K$ (Definition \ref{def dominating degree}), which roughly speaking counts the number of roots of $\phi$ in $\overline{K}$ which are pseudo-limits of $E$. 
Through this result, we show that, for each pseudo-monotone sequence $E$, the ring $V_E$ 
is a valuation domain of $K(X)$ extending $V$ (Theorem \ref{VE valuation domain}). We then describe the main properties of $V_E$ (residue field, value group and associated valuation) in Proposition \ref{prop:description}, and show that the image of a pseudo-convergent or a pseudo-divergent sequence under a rational function is eventually either  pseudo-convergent or pseudo-divergent (Proposition \ref{image pm under a rational function}), improving the analogous result of Ostrowski \cite[III, \S64, p. 371]{Ostrowski} on images of pseudo-convergent sequences under polynomial mappings.

In Section \ref{sect:fundamentalsatz}, we associate to each extension $W$ a subset $\inslim(W)$ of $K$ (which corresponds to the notion of pseudo-limit of a pseudo-monotone sequence) and show that if $K$ is algebraically closed, then $\inslim(W)$ (if nonempty) uniquely determines $W$ (Proposition \ref{prop:uglim}). 
In Section \ref{section:equivalence}, we use the results of the previous section to completely describe (for any field $K$) when two different pseudo-monotone sequences of $K$ give rise to the same associated extension of $V$ to $K(X)$. Subsequently, in Section \ref{generalized Fundamentalsatz} we give the proof of the aforementioned main Theorem \ref{teor:fundamentalsatz}.

In the final Section \ref{geometrical}, we illustrate the different containments which may occur among the valuation domains $V_E$ of $K(X)$. We conclude with a modern proof of Ostrowski's Fundamentalsatz (Theorem \ref{Ostrowski fundamentalsatz}).

\section{Background and notation}\label{notation}
For an extension $U$ or $W$ of $V$ to a field $F$ containing $K$, we denote the associated valuation with the corresponding lower case letter (i.e., $u$ or $w$, respectively). We recall that an extension $V\subset U$ is \emph{immediate} if $U$ and $V$ have the same value group and same residue field.  We denote by $\widehat{K}$ and $\widehat{V}$, respectively, the \emph{completion} of $K$ and $V$ with respect to the topology induced by the valuation $v$. The elements of $\widehat{K}$ can be constructed as limits of Cauchy sequences $\{a_\nu\}_{\nu\in\Lambda}$, where $\Lambda$ is a well-ordered set; $\Lambda$ is not necessarily countable, but can be considered of cardinality equal to the cofinality of the ordered set $\Gamma_v$. See for example \cite[Section 2.4]{prestelengler} for the details of the construction. For a sequence $\{s_\nu\}_{\nu\in\Lambda}$ of elements in $K$, the set of indices $\Lambda$ will always be a well-ordered set without a maximum.

\subsection{Monomial valuations}\label{monomial}

We recall the definition of monomial valuations, a standard way of extending a valuation $v$ of $K$ to $K(X)$.
\begin{Def}
Let $\Gamma$ be a totally ordered group containing $\Gamma_v$, and let  $\alpha\in K$ and $\delta\in\Gamma$. For every polynomial $f(X)=a_0+a_1(X-\alpha)+\ldots+a_n(X-\alpha)^n\in K[X]$, define
\begin{equation*}
v_{\alpha,\delta}(f)=\inf\{v(a_i)+i\delta \mid i=0,\ldots,n\},
\end{equation*}
and, for a rational function $\phi=f/g$ (with $f,g$ polynomials), define $v_{\alpha,\delta}(\phi)=v_{\alpha,\delta}(f)-v_{\alpha,\delta}(g)$. Then, $v_{\alpha,\delta}$ is a valuation on $K(X)$, and it is called \emph{monomial valuation} \cite[Chapt. VI, \S. 10, Lemma 1]{bourbaki-inglese}. We denote by $V_{\alpha,\delta}$ the associated valuation domain of $K(X)$. 
\end{Def}
For example, the Gaussian extension $v_G=v_{0,0}$ of $v$, defined as  $v_G(\sum_{i\geq0}a_i X^i)=\inf_i\{v(a_i)\}$, is a monomial valuation. In general, $v_{\alpha,\delta}$ is residually transcendental over $v$ (i.e., the residue field of $V_{\alpha,\delta}$ is transcendental over the residue field of $V$) if and only if  $\delta$ is torsion over $\Gamma_v$ \cite[Lemma 3.5]{PerPrufer}. Furthermore, every residually transcendental extension of $V$ can be written as $W'\cap K(X)$, where $W'$ is a monomial valuation domain of $\overline{K}(X)$ with respect to an extension $w$ of $v$ to $\overline K$  (\cite{AP,APZ1}).

\subsection{Divisorial ideals}\label{divisorial ideals}
Let $V$ be a valuation domain with maximal ideal $M$, and let $\mathcal{F}(V)$ be the set of fractional ideals of $V$. The \emph{$v$-operation} (or \emph{divisorial closure}) on $V$ is the map sending each $I\in\mathcal{F}(V)$ to the ideal $I^v$ equal to the intersection of all principal fractional ideals containing it; equivalently, $I^v=(V:(V:I))$, where, for a fractional ideal $I$ of $V$, we set $(V:I)=\{x\in K\mid xI\subseteq V\}$ \cite[Theorem 34.1]{Gilmer}. If $I=I^v$, we say that $I$ is a \emph{divisorial ideal}. 

If the maximal ideal $M$ of $V$ is principal, then each fractional ideal $I$ of $V$ is divisorial; on the other hand, if $M$ is not principal, then (see for example \cite[\textsection 34, Exercise 12, p. 431]{Gilmer})
\begin{equation*}
I^v=\begin{cases}
cV, & \text{if~}I=cM\text{~for some~}c\in K\\
I, & \text{otherwise}.
\end{cases}
\end{equation*}

We say that $I$ is \emph{strictly divisorial} if $I$ is equal to the intersection of all principal fractional ideals \emph{properly} containing it; in particular, each strictly divisorial ideal is divisorial. We now characterize these ideals.
\begin{lemma}\label{lemma:strictlydiv}
$I$ is not strictly divisorial if and only if $I=cM$ for some $c\in K$.
\end{lemma}
\begin{proof}
Suppose first that $I$ is not principal. Then, $I$ is strictly divisorial if and only if it is divisorial; furthermore, $I$ is not divisorial if and only if $I=cM$ for some $c\in K$ and $M$ is not principal, by the above remark. Hence, the claim holds in this case.

Suppose that $I=c'V$ is principal: if also $I=cM$ for some $c$, then $cV$ is the minimal principal ideal properly containing $I$, and $I$ is not strictly divisorial. Conversely, if $I$ is not strictly divisorial,  then there is a minimal principal ideal $cV$ properly containing $I$; this implies that $c'/c$ is the generator of the maximal ideal of $V$, and so $I=cM$.
\end{proof}

\subsection{Pseudo-monotone sequences}\label{pmonotone sequences}
The central concept of the paper is the following, which along with Ostrowski's notion of pseudo-convergent sequence includes also other two related notions introduced by Chabert in \cite{ChabPolClo}.
\begin{Def}
Let $E=\{s_\nu\}_{\nu\in\Lambda}\subset K$ be a sequence. We say that the sequence $E$ is:
\begin{itemize}
\item[-] \emph{pseudo-convergent} if $v(s_{\rho}-s_\nu)<v(s_{\sigma}-s_{\rho})$ for all $\nu<\rho<\sigma\in\Lambda$;
\item[-] \emph{pseudo-divergent} if $v(s_{\rho}-s_\nu)>v(s_{\sigma}-s_{\rho})$ for all $\nu<\rho<\sigma\in\Lambda$;
\item[-] \emph{pseudo-stationary} if $v(s_\nu-s_\mu)=v(s_{\nu'}-s_{\mu'})$ for all $\nu\neq\mu\in\Lambda$, $\nu'\neq\mu'\in\Lambda$.
\end{itemize}
If $E$ satisfies any of these definitions, we say that $E$ is a \emph{pseudo-monotone sequence} (\cite{PerPrufer}). We say that $E$ is \emph{strictly pseudo-monotone} if $E$ is either pseudo-convergent or pseudo-divergent. If $E$ and $F$ are two pseudo-monotone sequences that are either both pseudo-convergent, both pseudo-divergent or both pseudo-stationary we say that $E$ and $F$ are \emph{of the same kind}.
\end{Def}

We note that Ostrowski's and Chabert's original definitions required the above condition to be valid only for all $\nu$ large enough. Instead, we adopt Kaplansky's convention that the condition is valid for all $\nu$, both since it is not restrictive for our purposes (see Definition \ref{def:VE}) and in view of the following remark. If $E=\{s_\nu\}_{\nu\in\Lambda}$ is a sequence in $K$ and $E'=\{s_\nu\}_{\nu\geq N}$ is pseudo-monotone for some $N\in\Lambda$, we say that $E$ is \emph{eventually pseudo-monotone} (and analogously for eventually pseudo-convergent, pseudo-divergent and pseudo-stationary).

\begin{oss}\label{oss:rigidity}
Strictly pseudo-monotone sequences are ``rigid'', in the sense that, given a set $E$, there is at most one way to index $E$ to make it pseudo-monotone. Indeed, if the indexing $\{s_\nu\}_{\nu\in\Lambda}$ makes $E$ pseudo-convergent, then the equality  $v(s_\nu-s_\mu)=v(s_\nu-s_{\mu'})$ (for $\mu\neq\mu'$) implies that both $\mu$ and $\mu'$ are greater than $\nu$; thus, the elements of $E$ that appear before $s_\nu$ are exactly the $t$ such that $v(s_\nu-t)\neq v(s_\nu-t')$ for all $t\neq t'$, and this condition depends only on the set $E$. In the same way, if $E$ is pseudo-divergent, then the elements of $E$ appearing after $s_\nu$ are the $t$ such that $v(s_\nu-t)\neq v(s_\nu-t')$ for all $t\neq t'$. In particular, if $E=\{s_\nu\}_{\nu\in\Lambda}$ and $F=\{t_\nu\}_{\nu\in\Lambda}$ are two strictly pseudo-monotone sequences that are equal as sets, then $s_\nu=t_\nu$ for every $\nu\in\Lambda$.

On the other hand, pseudo-stationary sequences are ``flexible'': any permutation of $E=\{s_\nu\}_{\nu\in\Lambda}$ is again pseudo-stationary. For this reason, it may be more apt to call them ``pseudo-stationary \emph{sets}'', but we will continue to treat them as sequences for analogy with the strictly pseudo-monotone case.
\end{oss}

In this paper, we shall treat pseudo-monotone sequences in a general framework in order to build extensions of the valuation domain $V$ to the field of rational functions $K(X)$, and to give theorems valid for all kind of such sequences. However, there are slight differences in how the main concepts concerning pseudo-monotone sequences (for example the breadth ideal, the pseudo-limit and the gauge) are defined in each of the three cases; hence, we shall describe them separately.

\subsubsection{Pseudo-convergent sequences}\label{pcv}

Let $E=\{s_\nu\}_{\nu\in\Lambda}$ be a pseudo-convergent sequence in $K$. Then, if $\nu$ is fixed, the value $v(s_{\rho}-s_{\nu})$, for $\rho>\nu$, does not depend on $\rho$. We denote by $\delta_\nu\in\Gamma_v$ this value; the sequence $\{\delta_\nu\}_{\nu\in\Lambda}$ (which, by definition, is a strictly increasing sequence in $\Gamma_v$) is called the \emph{gauge} of $E$. 

The \emph{breadth ideal} $\Br(E)$ of $E$ is the set
\begin{align*}
\Br(E)= & \{x\in K\mid v(x)>\delta_\nu\text{~for all~}\nu\in\Lambda\};
\end{align*}
this set is always a fractional ideal of $K$. If $c_\nu=s_{\rho}-s_\nu$, for some $\rho>\nu$, then  $\Br(E)=\bigcap_{\nu\in\Lambda}c_\nu V$. If $\Br(E)$ is a principal ideal, say generated by an element $c\in K$, then $\delta_\nu$ converges to an element $\delta\in\Gamma_v$ (and, clearly, $v(c)=\delta$). When this happens, we call $\delta$ the \emph{breadth} of $E$. Note, however, that the breadth of a pseudo-convergent sequence may not always be defined; if $V$ has rank $1$ (that is, if $\Gamma_v$ can be embedded as a totally ordered group into $\insR$), then  $\delta_\nu$ always converges to an element $\delta\in\insR$, which may not belong to $\Gamma_v$. See \cite{PS} and \cite[Lemma 2.3]{PerPrufer} for this case.

An element $\alpha\in K$ is a \emph{pseudo-limit} of $E$ if $v(\alpha-s_\nu)=\delta_\nu$ for all $\nu\in\Lambda$ or, equivalently, if $v(\alpha-s_\nu)<v(\alpha-s_{\rho})$ for all $\nu<\rho\in\Lambda$. It also suffices that these conditions hold only for $\nu\geq N$, for some $N\in\Lambda$. If the gauge $\{\delta_\nu\}_{\nu\in\Lambda}$ is cofinal in $\Gamma_v$ (or, equivalently, if $E$ is a Cauchy sequence), then it is well-known that $E$ converges to a unique pseudo-limit $\alpha$ in the completion $\widehat K$, which in this case is called simply \emph{limit}.

Following Kaplansky \cite{kaplansky-maxfield}, we say that $E$ is of \emph{transcendental type} if $v(f(s_\nu))$ eventually stabilizes for every $f\in K[X]$; on the other hand, if $v(f(s_\nu))$ is eventually increasing for some $f\in K[X]$, we say that $E$ is of \emph{algebraic type}. As we have already remarked in \cite{PS}, it follows from the work of Kaplansky in \cite{kaplansky-maxfield} that a pseudo-convergent sequence $E\subset K$ is of algebraic type if and only if $E$ admits pseudo-limits in $\overline{K}$, with respect to some extension $u$ of $v$.  Note that any pseudo-convergent sequence satisfies either one of these two conditions, because the image of a pseudo-convergent sequence by a polynomial is an eventually pseudo-convergent sequence (see \cite[III, \S64, p. 371]{Ostrowski} or  Proposition \ref{image pm under a rational function} below).

\subsubsection{Pseudo-divergent sequences}\label{pdv}

Let $E=\{s_\nu\}_{\nu\in\Lambda}$ be a pseudo-divergent sequence in $K$. Symmetrically to the case of pseudo-convergent sequences, for a fixed $\nu$, we have that $v(s_{\rho}-s_{\nu})$ is constant for all $\rho<\nu$; if $\nu$ is not the minimum of $\Lambda$, we denote by $\delta_\nu\in\Gamma_v$ this value. The sequence $\{\delta_\nu\}_{\nu\in\Lambda}$  is a strictly decreasing sequence in $\Gamma_v$, called the \emph{gauge} of $E$.

The \emph{breadth ideal} $\Br(E)$ of $E$ is the set
\begin{align*}
\Br(E)= & \{x\in K\mid v(x)>\delta_\nu\text{~for some~}\nu\in\Lambda\};
\end{align*}
this set is a fractional ideal of $K$ if and only if the gauge of $E$ is bounded from below, while otherwise $\Br(E)=K$. In particular,  unlike in the pseudo-convergent case, $\Br(E)$ may  not be a fractional ideal. If for each non-minimal $\nu\in\Lambda$ we set $c_\nu=s_{\rho}-s_\nu$, for some $\rho<\nu$, then  $\Br(E)=\bigcup_{\nu\in\Lambda}c_\nu V$. Contrary to the case of a pseudo-convergent sequence, it is easily seen that the breadth ideal of a pseudo-divergent sequence is never a principal ideal. However, if $\delta_\nu\searrow\delta$, for some $\delta\in\Gamma_v$, then $\Br(E)=\{x\in K\mid v(x)>c\}=cM$, where $c\in K$ has value $\delta$. As in the case of a pseudo-convergent sequence, when this condition holds we call $\delta$ the \emph{breadth} of $F$.

An element $\alpha\in K$ is a \emph{pseudo-limit} of $E$ if $v(\alpha-s_\nu)=\delta_\nu$ for all (sufficiently large) $\nu\in\Lambda$ or, equivalently, if $v(\alpha-s_\nu)>v(\alpha-s_{\rho})$ for all (sufficiently large) $\nu<\rho\in\Lambda$. Every element of $E$ is a pseudo-limit of $E$: see \cite[\S 2.1.3]{PerPrufer} and Lemma \ref{pseudolimiti} below.

\subsubsection{Pseudo-stationary sequences}\label{pst}

Let $E=\{s_\nu\}_{\nu\in\Lambda}$ be a pseudo-stationary sequence in $K$. Note that the residue field of $V$ is necessarily infinite (see \cite[\S 2.1.2]{PerPrufer}). The element $\delta=v(s_\nu-s_\mu)\in \Gamma_v$, for $\nu\neq\mu$, is called the \emph{breadth} of $E$. In analogy with pseudo-convergent and pseudo-divergent sequences, we define the \emph{gauge} of $E$ to be the constant sequence $\{\delta_\nu=\delta\}_{\nu\in\Lambda}$.

The \emph{breadth ideal} $\Br(E)$ of $E$ is the set
\begin{align*}
\Br(E)= & \{x\in K\mid v(x)\geq\delta\};
\end{align*}
this set is always a principal fractional ideal of $K$, generated by any $c\in K$ whose value is $\delta$. In particular, we can take $c=s_{\nu'}-s_\nu$ for any $\nu'\neq\nu$.

An element $\alpha\in K$ is a \emph{pseudo-limit} of $E$ if $v(\alpha-s_\nu)=\delta$ for all sufficiently large $\nu\in\Lambda$ or, equivalently, if $v(\alpha-s_\nu)=\delta$ for all but at most one $\nu\in\Lambda$. As in the pseudo-divergent case, every element of $E$ is a pseudo-limit of $E$: see \cite[\S 2.1.2]{PerPrufer} and Lemma \ref{pseudolimiti} below.

\subsection{Pseudo-limits and the breadth ideal}
In general, if $E\subset K$ is a pseudo-monotone sequence, we denote the set of pseudo-limits of $E$ in $K$ by  $\inslim_E$ and the breadth ideal by $\Br(E)$ (or $\inslim_E^v$ and $\Br_v(E)$, respectively, if we need to underline the valuation). We will constantly use the following trivial remark: if $u$ is an extension of $v$ to an overfield $F$ of $K$, then $E$ is a pseudo-monotone sequence in the valued field $(F,u)$; in particular, $\mathcal L_F^u$ will denote the set of pseudo-limits of $E$ in the valued field $(F,u)$. We use the notation $\inslim_E$ and $\Br(E)$ also in the case $E$ is only eventually pseudo-monotone.

The first part of the next result generalizes the classical result of Kaplansky for pseudo-convergent sequences (\cite[Lemma 3]{kaplansky-maxfield}) to pseudo-monotone sequences. The proof is actually the same, but for the sake of the reader we give it here.
\begin{lemma}\label{pseudolimiti}
Let $E=\{s_\nu\}_{\nu\in\Lambda}\subset K$ be a pseudo-monotone sequence and let $\alpha\in K$ be a pseudo-limit of $E$. Then the set of pseudo-limits $\mathcal L_E$ of $E$ is equal to $\alpha+\Br(E)$. Moreover, $E\cap \mathcal{L}_E=\emptyset$ if $E$ is a pseudo-convergent sequence and $E\subset\mathcal{L}_E$ if $E$ is either pseudo-divergent or pseudo-stationary.
\end{lemma}
\begin{proof}
Let $\beta=\alpha+x$, for some $x\in \Br(E)$. If $E$ is either pseudo-convergent or pseudo-divergent, then it is easy to see that for any $\nu\in\Lambda$ we have
$$v(\beta-s_\nu)=v(\alpha-s_\nu+x)=v(\alpha-s_\nu)=\delta_\nu$$
so that $\beta$ is a pseudo-limit of $E$. If $E$ is pseudo-stationary, then we have $v(\beta-s_\nu)\geq \delta=v(s_\nu-s_\mu)=v(s_\nu-\beta+\beta-s_\mu)$ and therefore for at most one $\nu\in\Lambda$ we may have the strict inequality $v(\beta-s_\nu)>\delta$. So, also in this case $\beta$ is a pseudo-limit of $E$.

Conversely, if $\beta$ is a pseudo-limit of $E$, then $v(\alpha-\beta)=v(\alpha-s_\nu+s_\nu-\beta)\geq\delta_\nu$, so that $\alpha-\beta\in\Br(E)$, as we wanted to show.

We prove the last claim. If $E$ is a pseudo-convergent sequence,  then it is clear (both if $E$ is of algebraic type or of transcendental type). If the sequence $E$ is either  pseudo-divergent or pseudo-stationary, the claim is proved in \cite[\S 2.1.2 \& \S 2.1.3]{PerPrufer}.
\end{proof}

In particular, since pseudo-divergent and pseudo-stationary sequences always admit a pseudo-limit in $K$, in these cases there is no analogue of the notion of pseudo-convergent sequences of transcendental type.

The following result characterizes which fractional ideals of $V$ are breadth ideals for some pseudo-monotone sequence $E$ of $K$, and which cosets are the set of pseudo-limits for some pseudo-monotone sequence.

\begin{lemma}\label{lemma:carattbreadth}
Let $I$ be a fractional ideal of $V$ and let $\alpha\in K$; let $\inslim=\alpha+I$.
\begin{enumerate}[(a)]
\item $\inslim=\inslim_E$ for some pseudo-convergent sequence $E$ if and only if $I$ is strictly divisorial; in particular, if the maximal ideal of $V$ is not principal this happens if and only if $I$ is divisorial.
\item $\inslim=\inslim_E$ for some pseudo-divergent sequence if and only if $I$ is not principal.
\item If $V/M$ is infinite, $\inslim=\inslim_E$ for some pseudo-stationary sequence if and only if $I$ is principal.
\end{enumerate}
\end{lemma}
\begin{proof}
It is easily seen that, if $\inslim_E\neq\emptyset$ for some pseudo-monotone sequence $E=\{s_\nu\}_{\nu\in\Lambda}$, for every $\beta\in K$ the set $\beta+\inslim_E$ is the set of pseudo-limits of $\beta+E=\{\beta+s_\nu\}_{\nu\in\Lambda}$; hence, it is enough to prove the claims for $\alpha=0$. Furthermore, by Lemma \ref{pseudolimiti}, under this hypothesis we have $\inslim_E=\Br(E)$, and thus we only need to find which ideals are breadth ideals.

If $I=\Br(E)$ for some pseudo-convergent $E=\{s_{\nu}\}_{\nu\in\Lambda}$, for each $\nu$ let $c_{\nu}=s_{\rho}-s_{\nu}$, for some $\rho>\nu$; then $I=\bigcap_{\nu}c_\nu V$, and each $c_\nu V$ properly contains $I$. Therefore $I$ is a strictly divisorial ideal. Conversely, if $I=\bigcap_{a\in A} aV$, where for each $a\in A$ we have $I\subsetneq aV$, we can take a well-ordered subset $\{a_\nu\}_{\nu\in\Lambda}$ such that $I=\bigcap_\nu a_\nu V$ and $a_{\rho}V\subsetneq a_\nu V$ for all $\rho>\nu$; then, $\{a_\nu\}_{\nu\in\Lambda}$ is a pseudo-convergent sequence having 0 as a pseudo-limit and breadth ideal $I$. The last remark follows from Lemma \ref{lemma:strictlydiv}.

Likewise, if $I=\Br(E)$ for some pseudo-divergent $E=\{s_{\nu}\}_{\nu\in\Lambda}$, for each $\nu$ let $c_{\nu}=s_{\rho}-s_{\nu}$, for some $\rho<\nu$; then $I=\bigcup_{\nu}c_\nu V$, while if $I$ is not principal we can find a well-ordered sequence $E=\{a_\nu\}_{\nu\in\Lambda}\subset V$ which generates $I$ and such that $a_\nu V\subset a_{\rho}V$ for every $\nu<\rho$, so that $E$ is a pseudo-divergent sequence  and  $I$ is its breadth ideal.

If $I=\Br(E)$ for some pseudo-stationary sequence $E=\{s_{\nu}\}_{\nu\in\Lambda}$, then $I=(s_\nu-s_{\mu})V$, for any $\nu\not=\mu$; conversely, if $I=cV$, then we can find a well-ordered set $E=\{s_\nu\}_{\nu\in\Lambda}$ of distinct elements of valuation $v(c)$ whose cosets modulo $cM$ are different (because the residue field of $V$ is infinite); then, $E$ is pseudo-stationary with breadth ideal $E$.
\end{proof}

\section{Valuation domains associated to pseudo-monotone sequences}\label{valuations associated pmt}

Let $\phi\in K(X)$ be a rational function: if $\alpha\in\overline{K}$ is a zero or a pole of $\phi$, we say that $\alpha$ is a \emph{critical point} of $\phi$. We denote by $\Omega_\phi$ the multiset of critical points of $\phi$. Let $S=\{\alpha_1,\ldots,\alpha_k\}$ be a submultiset of $\Omega_\phi$. The \emph{weighted sum} of $S$ is the sum $\sum_{\alpha_i\in S}\epsilon_i$, where $\epsilon_i=1$ if $\alpha_i$ is a zero of $\phi$ and $\epsilon_i=-1$ if $\alpha_i$ is a pole of $\phi$. The \emph{$S$-part} of $\phi$ is the rational function $\phi_S(X)=\prod_{\alpha_i\in S}(X-\alpha_i)^{\epsilon_i}$, where $\epsilon_i$ is as above.

The following definition generalizes \cite[Definition 3.5]{PS} to pseudo-monotone sequences.
\begin{Def}\label{def dominating degree}
Let $E=\{s_\nu\}_{\nu\in\Lambda}$ be a pseudo-monotone sequence in $K$, let $u$ be an extension of $v$ to $\overline{K}$ and let $\phi\in K(X)$. The \emph{dominating degree} $\degdom_{E,u}(\phi)$ of $\phi$ with respect to $E$ and $u$ is the weighted sum of the elements of $\Omega_{\phi}$  which are pseudo-limits of $E$ with respect to $u$. 
\end{Def}

The next proposition is a generalization to pseudo-monotone sequences of \cite[Theorem 3.3]{PS}; in particular, it shows that the dominating degree does not depend on the chosen extension of $v$ to $\overline{K}$.

\begin{Prop}\label{values rat func on psm seq}
Let $E=\{s_\nu\}_{\nu\in\Lambda}\subset K$ be a pseudo-monotone sequence of gauge $\{\delta_\nu\}_{\nu\in\Lambda}$, and let $\phi\in K(X)$. Let $u$ be an extension of $v$ to $\overline{K}$ and let $\lambda=\degdom_{E,u}(\phi)$.  
Then there exist $\gamma\in\Gamma_v$ and $\nu_0\in\Lambda$ such that for each $\nu\geq\nu_0$ we have
$$v(\phi(s_\nu))=\lambda\delta_\nu+\gamma.$$
Furthermore, if $\beta\in\overline{K}$ is a pseudo-limit of $E$ with respect to $u$, then  $\gamma=u\left(\frac{\phi}{\phi_S}(\beta)\right)$, where $S$ is the set of critical points of $\phi$ which are pseudo-limits of $E$ with respect to $u$. 

Moreover, the dominating degree of $\phi$ does not depend on $u$; that is, if $u'$ is another extension of $v$ to $\overline{K}$, then $\degdom_{E,u}(\phi)=\degdom_{E,u'}(\phi)$.
\end{Prop}
\begin{proof}
If $E$ is a pseudo-convergent sequence, then the statement is the same as in \cite[Proposition 3.6]{PS}.

If the sequence $E$ is pseudo-divergent, then the proof is essentially the same as when $E$ is pseudo-convergent:  let $\beta\in K$ be a pseudo-limit of $E$ and let $\Delta=\Delta_E$ be the least final segment of $\insQ\Gamma_v$ containing the gauge of $E$ (if $\Br(E)=K$, just take $\Delta=\Gamma_v$). Take $\tau\in\Gamma_v\cap\Delta$ such that $C=\{s\in\overline K \mid u(s-\beta)\in\Delta\cap(-\infty,\tau)\}$ contains no critical points of $\phi$. 
Then, $s_\nu\in C$ for all large $\nu$ and, by construction, the weighted sum of the subset $S$ of $\Omega_\phi$ of those elements $\alpha$ such that $u(\alpha-\beta)> \Delta\cap(-\infty,\tau)$ is exactly $\lambda=\degdom_{E,u}(\phi)$. Therefore, we can apply  \cite[Theorem 3.3]{PS} to the convex set $\Delta\cap(-\infty,\tau)$, and thus there is a $\nu_0\in\Lambda$ such that for each $\nu\geq \nu_0$ we have
$$v(\phi(s_\nu))=\lambda v(s_\nu-\beta)+\gamma=\lambda\delta_\nu+\gamma,$$
where $\gamma=u\left(\frac{\phi}{\phi_S}(\beta)\right)$, as in the statement of the proposition, again by \cite[Theorem 3.3]{PS}. Since $v(\phi(s_\nu))\in\Gamma_v$ and does not depend on $\beta$, the same happens for $\gamma$. For the final claim the proof is analogous to \cite[Proposition 3.6(c)]{PS}.

If $E$ is pseudo-stationary, we cannot apply directly \cite[Theorem 3.3]{PS}, but the same general method works: let $\phi\in K(X)$ and write $\phi(X)=c\prod_{i=1}^n(X-\alpha_i)^{\epsilon_i}$, where $c\in K$, $\alpha_i\in\overline{K}$ and $\epsilon_i\in\{1,-1\}$. Let $u$ be a fixed extension of $v$ to $\overline{K}$, let $\beta\in K$ be a pseudo-limit of $E$ and let $S$ be the multiset of critical points of $\phi$ which are pseudo-limits of $E$  with respect to $u$. If $\alpha\in\Omega_{\phi}\setminus S$, then $u(s_\nu-\alpha)=u(\beta-\alpha)<\delta$ for all sufficiently large $\nu\in\Lambda$; on the other hand, if $\alpha\in S$, then there is at most one $\nu$ (say $\nu_0$) such that $u(s_{\nu_0}-\alpha)>\delta$, while $u(s_\nu-\alpha)=\delta$ for all $\nu\neq\nu_0$. Hence, for all large $\nu$ we have $u(s_\nu-\alpha)=\delta$. Note that, if $\alpha\notin S$, then $u(\beta-\alpha)$ does not depend on the chosen pseudo-limit $\beta$ of $E$. In particular, $u(s_\nu-\alpha)\leq \delta$ and equality holds if and only if $\alpha$ is a pseudo-limit of $E$, in complete analogy with \cite[Remark 4.7(a)]{PS}. Now, let $\lambda$ be the weighted sum of $S$ (which is equal to $\degdom_{E,u}(\phi)$) and $\gamma=u\left(\frac{\phi}{\phi_S}(\beta)\right)$: then, for all large $\nu$, $s_\nu$ is not a critical point of $\phi$ and we have
\begin{align*}
v(\phi(s_\nu))&=v(c)+\sum_{\alpha\in S}\epsilon_i u(s_\nu-\alpha)+\sum_{\alpha\in\Omega_{\phi}\setminus S}\epsilon_i u(s_\nu-\alpha)=\lambda\delta+\gamma
\end{align*}
It is clear as before that $\gamma\in\Gamma_v$ and does not depend on the chosen pseudo-limit $\beta$ of $E$, by the above remark. To conclude, we only need to prove that the dominating degree of $\phi$ with respect to a pseudo-stationary sequence $E$ does not depend on the extension of $v$ to $\overline{K}$. Let $u,u'$ be two extensions of $v$ to $\overline{K}$. By Lemma \ref{pseudolimiti}, it follows that $\mathcal{L}_E=s+cV$, where a pseudo-limit $s$ of $E$ can be chosen in $K$ and $c\in K$ has value $\delta_E$. Now, by the same Lemma we also have that $\mathcal{L}_E^u=s+cU$ and $\mathcal{L}_E^{u'}=s+cU'$; in particular, $\mathcal{L}_E^u$ and $\mathcal{L}_E^{u'}$ are conjugate under the action of the Galois group of $\overline{K}$ over $K$. It is then clear that $\Omega_{\phi}\cap \mathcal{L}_E^u$ and $\Omega_{\phi}\cap \mathcal{L}_E^{u'}$ are conjugate too, so $\degdom_{E,u}(\phi)=\degdom_{E,u'}(\phi)$, as desired
\end{proof}

Note that by Proposition \ref{values rat func on psm seq} we may drop the suffix $u$ in the dominating degree of a rational function. However, note that given a pseudo-monotone sequence $E\subset K$ without pseudo-limits in $K$, different extensions of $v$ to $\overline K$ give rise to different set of pseudo-limits, which are conjugate under the action of the Galois group of $\overline{K}$ over $K$.

Moreover, if $E=\{s_\nu\}_{\nu\in\Lambda}$ is a pseudo-stationary sequence and $\phi\in K(X)$, the values of $\phi$ on $E$ are eventually constant, namely $v(\phi(s_\nu))=\lambda\delta+\gamma$, where $\lambda=\degdom_E(\phi)$, $\delta=\delta_E$ and $\gamma\in \Gamma_v$, for all sufficiently large $\nu$.

\begin{Def}\label{def:VE}
Let $E=\{s_\nu\}_{\nu\in\Lambda}\subset K$ be a pseudo-monotone sequence. We define
$$V_E=\{\phi\in K(X) \mid \phi(s_\nu)\in V, \textnormal{ for all sufficiently large }\nu\in\Lambda\}.$$
\end{Def}

\begin{teor}\label{VE valuation domain}
Let $E=\{s_\nu\}_{\nu\in\Lambda}\subset K$ be a pseudo-monotone sequence. Then $V_E$ is a valuation domain with maximal ideal $$M_E=\{\phi\in K(X) \mid \phi(s_\nu)\in M, \text{ for all sufficiently large }\nu\in\Lambda\}.$$
\end{teor}
\begin{proof}
The proof is exactly as the one of \cite[Theorem 3.8]{PS}, but we repeat it here for completeness.

The set $V_E$ is a ring since if $\phi(s_\nu),\psi(s_\nu)\in V$ for all sufficiently large $\nu$, then also $(\phi+\psi)(s_\nu)$ and $(\phi\psi)(s_\nu)$ are eventually in $V$.

Let $\phi\in K(X)$. By Proposition \ref{values rat func on psm seq}, we have $v(\phi(s_\nu))=\lambda\delta_\nu+\gamma$, for all $\nu\in\Lambda$ sufficiently large, for some $\lambda\in\Z$ and $\gamma\in\Gamma_v$. In particular, the values of $\phi$ over $E$ are either eventually positive, eventually negative or eventually constant, so either $\phi(s_\nu)\in V$ or $\phi(s_\nu)^{-1}=\phi^{-1}(s_\nu)\in V$ (in both cases for all $\nu\in\Lambda$ sufficiently large), which shows that $V_E$ is a valuation domain.

The claim about the maximal ideal of $V_E$ follows immediately.
\end{proof}

We call $V_E$ the \emph{extension of $V$ associated to the pseudo-monotone sequence $E$}. Note that, if $E$ is a pseudo-convergent sequence and its gauge is cofinal in $\Gamma_v$ (or, equivalently, $E$ is a Cauchy sequence), then $V_E=V_{\alpha}=\{\phi\in K(X) \mid \phi(\alpha)\in\widehat V\}$, where $\alpha$ is the (unique) limit of $E$ in the completion $\widehat{K}$. See \cite{PerTransc} for a study of these valuation domains.

The main properties of the valuation domain $V_E$ and its associated valuation $v_E$ are summarized in Proposition \ref{prop:description} below, which is a generalization of \cite[Proposition 3.11]{PS}. We need to introduce another definition.

\begin{Def}\label{def:PE}
Let $E\subset K$ be a pseudo-monotone sequence. We denote by $\mathcal{P}_E$ the set of the irreducible monic polynomials $p\in K[X]$ which have at least one root in $\overline K$ which is a pseudo-limit of $E$ (with respect to some extension of $v$ to $\overline{K}$), or, equivalently, such that $\degdom_E(p)>0$.
\end{Def}

We note that $\mathcal{P}_E$ is nonempty if and only if $E$ has a pseudo-limit in $\overline{K}$; that is, $\mathcal{P}_E$ is empty if and only if $E$ is a pseudo-convergent sequence of transcendental type. If $E$ is a pseudo-convergent sequence of algebraic type which is also a Cauchy sequence, then $\mathcal{P}_E$ contains a unique element, namely the minimal polynomial of the (unique) limit of $E$ in $\widehat{K}$ (and by Lemma \ref{pseudolimiti} this is the only case in which $\mathcal P_E$ has only one element).
\begin{lemma}\label{minimal degree in PE}
Let $E$ be a strictly pseudo-monotone sequence having a pseudo-limit in $\overline{K}$, and let $p\in K[X]$. Then:
\begin{enumerate}[(a)]
\item\label{minimal degree in PE:inGammav} $v_E(p)\notin\Gamma_v$ if and only if some irreducible factor of $p$ is in $\mathcal{P}_E$;
\item\label{minimal degree in PE:torsion} if $v_E(p)\notin\Gamma_v$, then $v_E(p)$ is not torsion over $\Gamma_v$;
\item\label{minimal degree in PE:mindeg} if $p_1,p_2\in\mathcal P_E$ are of minimal degree, then $v_E(p_1)=v_E(p_2)$.
\end{enumerate}
\end{lemma}
\begin{proof}
Let $p\in K[X]$. Then, $v_E(p)=v(t)$ for some $t\in K$ if and only if $v(t)=v(p(s_\nu))=\degdom_E(p)\delta_\nu+\gamma$ for all $\nu$ sufficiently large (Proposition \ref{values rat func on psm seq}); since $E$ is strictly pseudo-monotone, it follows that $v_E(p)\in\Gamma_v$ if and only if $\degdom_E(p)=0$. Since $\degdom_E(q_1\cdots q_n)=\sum_i\degdom_E(q_i)$, \ref{minimal degree in PE:inGammav} follows. 

\ref{minimal degree in PE:torsion} is a consequence of the previous point applied to the powers $p^n$ of $p$.

Finally, if $p_1,p_2\in\mathcal P_E$ are polynomials of minimal degree, then $p_1-p_2=r$ for some $r\in K[X]$ of lower degree, because $p_1,p_2$ are monic; by minimality, no factor of $r$ belongs to $\mathcal{P}_E$, and so $v_E(r)\in\Gamma_v$. Hence, it must be $v_E(p_1)=v_E(p_2)$ (otherwise $v_E(p_1-p_2)=\min\{v_E(p_1),v_E(p_2)\}$ which is not in $\Gamma_v$), and \ref{minimal degree in PE:mindeg} holds.
\end{proof}

\begin{Prop}\label{prop:description}
Let $E=\{s_\nu\}_{\nu\in\Lambda}\subset K$ be a pseudo-monotone sequence. If $\mathcal P_E$ is nonempty, we let $\Delta_E=v_E(p)$ for some $p\in\mathcal{P}_E$ of minimal degree.
\begin{enumerate}[(a)]
\item\label{prop:description:alg} If $E$ is either  pseudo-convergent of algebraic type or  pseudo-divergent, then  $\Gamma_{v_E}=\Delta_E\Z\oplus\Gamma_v$ (as groups) and $V_E/M_E\cong V/M$.
\item\label{prop:description:trasc} If $E$ is pseudo-convergent of transcendental type, then $V\subset V_E$ is immediate.
\item\label{prop:description:pstaz} If $E$ is  pseudo-stationary, then $\Gamma_{v_E}=\Gamma_v$ and $V_E/M_E$ is a purely transcendental extension of $V/M$: more precisely, $V_E/M_E=V/M(t)$, where $t$ is the residue of $\frac{X-\alpha}{c}$ modulo $M_E$, where $c\in K$ satisfies $v(c)=\delta_E$ and $\alpha\in\inslim_E$.
\item\label{prop:description:mindeg} If $E$ is not pseudo-convergent of transcendental type, then $\Gamma_{v_E}=\langle \Gamma_v,\Delta_E\rangle$. Furthermore, $\Delta_E$ does not depend on $p$ and, if $E$ is a pseudo-stationary sequence, $\Delta_E=\delta_E$.
\item\label{prop:description:monomiali} If $E$ has a pseudo-limit $\beta\in K$, then $v_E=v_{\beta,\Delta_{E}}$.
\end{enumerate}
\end{Prop}
\begin{proof}
\ref{prop:description:alg} In both cases, $E$ has a pseudo-limit in $\overline{K}$ with respect to some extension of $v$, and so $\mathcal{P}_E\neq\emptyset$.
Fix a polynomial $p\in\mathcal{P}_E$ of minimal degree, and let $\Delta_E=v_E(p)$, which does not depend on $p$ and is not torsion over $\Gamma_v$  by Lemma \ref{minimal degree in PE}. For every $q\in K[X]$, we can write $q=r_0+r_1p+r_2p^2+\cdots+r_np^n$, for some (uniquely determined) $r_0,\ldots,r_n\in K[X]$ such that $\deg r_i<\deg p$. Since $\Delta_E$ is not torsion over $\Gamma_v$ and $v_E(r_i)\in\Gamma_v$ for each $i$ by minimality of the degree of $p$, we have
\begin{equation*}
v_E(r_ip^i)=v_E(r_i)+i\Delta_E\neq v_E(r_j)+j\Delta_E=v_E(r_jp^j)
\end{equation*}
for every $i\neq j$; therefore, $v_E(q)=\min\{v_E(r_0),v_E(r_1p),\ldots,v_E(r_np^n)\}$, and in particular $v_E(q)\in\Gamma_v\oplus\Delta_E\insN$. Hence, $\Gamma_{v_E}=\Gamma_v\oplus\Delta_E\insZ$.

We now show that $V_E/M_E=V/M$. If $E$ has a pseudo-limit $\alpha$ in $K$, then as in  \cite[Proposition 3.11]{PS} by Lemma \ref{minimal degree in PE} we have $v_E=v_{\alpha,\Delta_E}$ and by \cite[Chap. VI, \S 10, 1., Proposition 1]{bourbaki-inglese} $V_E$ and $V$ have the same residue field. Suppose instead $\inslim_E=\emptyset$ (in particular, $E$ must be a pseudo-convergent sequence, by Lemma \ref{pseudolimiti}), and let $\phi$ be a unit of $V_E$. Let $u$ be an extension of $v$ to $\overline{K}$ and let $\alpha\in\inslim_E^u$. Then, the residue field of $U_E$ is equal to the residue field of $U$ (by the previous case); hence, there is a unit $\beta$ of $U$ such that $\phi-\beta\in M_{U_E}$. Thus, $\phi(s_\nu)\in\beta+M_U$ for all $\nu$ bigger or equal than some $N\in\Lambda$.

Since $\phi$ is a unit of $V_E$, $\phi(s_\nu)$ is a unit of $V$ for all large $\nu$; without loss of generality, for $\nu\geq N$. Let $a$ be such that $\phi(s_N)\in a+M$: then, for every $\nu>N$ we have $\phi(s_\nu)-\phi(s_N)\in M_U\cap V=M$, and thus also $\phi(s_\nu)\in a+M$. Hence, the image of $\phi$ is in $V/M$, and so $V/M=V_E/M_E$. The claim is proved.

\ref{prop:description:trasc} This follows from Kaplansky's results in \cite{kaplansky-maxfield}.

\ref{prop:description:pstaz} Suppose that $E$ is a pseudo-stationary sequence. It is clear that, without loss of generality, we may suppose that $K$ is algebraically closed. In order to prove the claim, by \cite[\S 11, IV, p. 366]{Ostrowski} it is sufficient to show that $v_E(X-\alpha-\beta)=\min\{v_E(X-\alpha),v(\beta)\}$ for each $\beta\in K$. By Proposition \ref{values rat func on psm seq}, we have $v_E(X-\alpha)=\delta$. If $\delta\not=v(\beta)$ we are done. If $\delta=v(\beta)$, then by Lemma \ref{pseudolimiti}, $\alpha+\beta$ is a pseudo-limit of $E$, so again by Proposition \ref{values rat func on psm seq} we have $v_E(X-\alpha-\beta)=\delta$.

\ref{prop:description:mindeg} For pseudo-convergent sequences of algebraic type or pseudo-divergent sequences the claim follows from the proof of part \ref{prop:description:alg}. For a  pseudo-stationary sequence $E$, $\Delta_E=v_E(X-\alpha)=\delta_E$ for all pseudo-limits $\alpha\in\inslim_E$, and we are done. \ref{prop:description:monomiali} follows in the same way.
\end{proof}

The next proposition constitutes an important generalization of \cite[Lemma 5]{kaplansky-maxfield} and \cite[III, \S64, p. 371]{Ostrowski}, which says that the image under a polynomial of a pseudo-convergent sequence is an eventually pseudo-convergent sequence.

\begin{Prop}\label{image pm under a rational function}
Let $E=\{s_\nu\}_{\nu\in\Lambda}\subset K$ be a strictly pseudo-monotone sequence and let $\phi\in K(X)$ be non-constant. Then $\phi(E)=\{\phi(s_\nu)\}_{\nu\in\Lambda}$ is an eventually strictly pseudo-monotone sequence, which is of the same kind of $E$ if $\degdom_E(\phi)>0$, and not of the same kind if $\degdom_E(\phi)<0$; if $\degdom_E(\phi)\not=0$, then $\mathcal{L}_{\phi(E)}=\Br(\phi(E))$. Furthermore, if $\phi(E)$ is eventually pseudo-convergent, then $\phi(X)$ is a pseudo-limit of $\phi(E)$ with respect to $v_E$.
\end{Prop}
\begin{proof}
Let $\lambda=\degdom_E(\phi)$. Suppose first that $\lambda>0$ and $E$ is a pseudo-convergent sequence. By Proposition \ref{values rat func on psm seq} we have $v(\phi(s_\nu))=\lambda\delta_\nu+\gamma<v(\phi(s_\mu))=\lambda\delta_\mu+\gamma$ for all $\nu<\mu$ sufficiently large (say greater than some $\nu_0\in\Lambda$), which shows that $\phi(E)$ is an eventually pseudo-convergent sequence with gauge $\{\lambda\delta_\nu+\gamma\}_{\nu\in\Lambda}$. Since $v(\phi(s_\nu))$ increases, $0$ is a pseudo-limit of $\phi(E)$, and thus by Lemma \ref{pseudolimiti} we have the equality $\mathcal{L}_{\phi(E)}=\Br(\phi(E))$. 
Since $v(\phi(s_\rho))>v(\phi(s_\nu))$ if $\rho>\nu$ (sufficiently large), we have $v_E\left(\frac{\phi(X)}{\phi(s_\nu)}\right)> 0$ for all $\nu$ sufficiently large; hence, eventually, $v_E(\phi(X)-\phi(s_\nu))=v_E(\phi(s_\nu))$, and in particular $\{v_E(\phi(X)-\phi(s_\nu))\}_{\nu\in\Lambda}$ is strictly increasing. Hence, $\phi(X)$ is a pseudo-limit of $\phi(E)$.

If $\lambda>0$ and $E$ is a pseudo-divergent sequence, then as above $\phi(E)$ is eventually pseudo-divergent. If $\lambda<0$, then in the same way we can prove that $\phi(E)$ is strictly pseudo-monotone, not of the same kind of $E$, and $\phi(X)$ is a pseudo-limit of $\phi(E)$ with respect to $v_E$.

\medskip

Suppose now that $\lambda=0$ and  $E$ is a  pseudo-convergent sequence. Without loss of generality, we may also suppose  that $K=\overline K$. Let $\phi(X)=p(X)/q(X)$, where $p,q\in K[X]$. Since $K$ is algebraically closed, we can write $q(X)=q_1(X)q_2(X)$ in such a way that all zeros of $q_1$ are pseudo-limits of $E$ while no zero of $q_2$ is a pseudo-limit of $E$ (if $E$ has no pseudo-limits, then $q(X)=q_2(X)$ and $q_1(X)=1$). In particular, $\deg q_1=\degdom_E(q_1)$. Dividing $p$ by $q_1$, we have
\begin{equation*}
\phi(X)=\frac{p(X)}{q(X)}=\frac{a(X)q_1(X)+b(X)}{q(X)}=\frac{a(X)}{q_2(X)}+\frac{b(X)}{q(X)},
\end{equation*}
where $a,b\in K[X]$ and $\deg b<\deg q_1$. The rational function $\phi_2(X)=\frac{b(X)}{q(X)}$ has dominating degree
\begin{equation*}
\degdom_E(\phi_2)=\degdom_E(b)-\degdom_E(q_1)\leq\deg b-\deg q_1<0,
\end{equation*}
and thus, by the previous part of the proof, $\{\phi_2(s_\nu)\}_{\nu\in\Lambda}$ is an eventually pseudo-divergent sequence.

Consider now $\phi_1(X)=\frac{a(X)}{q_2(X)}$. If $E$ has a pseudo-limit in $K=\overline{K}$, let $\alpha\in\inslim_E$. If not, then $E$ is a pseudo-convergent sequence of transcendental type, and we can extend $v$ to a transcendental extension $K(z)$ of $K$ such that $z$ is a pseudo-limit of $E$ (\cite[Theorem 2]{kaplansky-maxfield}), and we set $\alpha=z$; with a slight abuse of notation, we still denote by $v$ this extension to $K(z)$. Note that in any case  $q_2(\alpha)\neq 0$ since $\degdom_E(q_2)=0$. Consider the following  rational function over $K(\alpha)$:
\begin{equation*}
\psi(X)=\phi_1(X)-\phi_1(\alpha)=\frac{a(X)q_2(\alpha)-a(\alpha)q_2(X)}{q_2(\alpha)q_2(X)}.
\end{equation*}
Since $\psi(\alpha)=0$, the dominating degree of the numerator of $\psi$ is positive; on the other hand, $\degdom_E(q_2(\alpha)q_2)=\degdom_E(q_2)=0$. Hence, $\degdom_E(\psi)>0$, and by the previous part of the proof $\{\psi(s_\nu)\}_{\nu\in\Lambda}$ is an eventually pseudo-convergent sequence in $K(\alpha)$. Thus, also $\{\phi_1(s_\nu)\}_{\nu\in\Lambda}=\{\psi(s_\nu)+\phi_1(\alpha)\}_{\nu\in\Lambda}$ is eventually pseudo-convergent in $K(\alpha)$; however, $\phi_1(s_\nu)\in K$ for every $\nu$, and thus $\{\phi_1(s_\nu)\}_{\nu\in\Lambda}$ is a  eventually pseudo-convergent sequence in $K$.

By definition, $\phi(s_\nu)=\phi_1(s_\nu)+\phi_2(s_\nu)$ and, by the previous points, the sequences $\{\phi_1(s_\nu)\}_{\nu\in\Lambda}$ and $\{\phi_2(s_\nu)\}_{\nu\in\Lambda}$ are eventually pseudo-convergent and eventually pseudo-divergent, respectively. In particular, for large $\nu$, $v(\phi_1(s_{\rho})-\phi_1(s_\nu))$, $\rho>\nu$, is increasing and $v(\phi_2(s_{\rho})-\phi_2(s_\nu))$, $\rho>\nu$, is decreasing; it follows that $v(\phi(s_{\rho})-\phi(s_\nu))$, $\rho>\nu$, is eventually equal to one of the two. Hence, $\phi(s_\nu)$ is eventually strictly pseudo-monotone, as claimed.

Suppose in particular that $\phi(E)$ is eventually  pseudo-convergent: then, 
\begin{equation*}
v_E(\phi(X)-\phi(s_\nu))=v_E((\phi_1(X)-\phi_1(s_\nu))+(\phi_2(X)-\phi_2(s_\nu))).
\end{equation*}
By the case $\lambda>0$, we have $v_E((\phi_1(X)-\phi_1(s_\nu))=v_E(\phi_1(s_\nu))$ for all large $\nu$. On the other hand, since $\phi(E)$ is pseudo-convergent we have $v_E(\phi_1(s_\nu))<v_E(\phi_2(s_\rho))$ for all large $\nu<\rho$; in particular, we also have $v_E(\phi_2(X))\geq v_E(\phi_1(X))$ and so $v_E(\phi_2(X)-\phi_2(s_\nu))$ is bigger than both $v_E(\phi_1(X))$ and $v_E(\phi_1(s_n))$. Hence, 
\begin{equation*}
v_E(\phi(X)-\phi(s_\nu))=v_E(\phi_1(X)-\phi_1(s_\nu))=v_E(\phi_1(s_\nu)),
\end{equation*}
which is eventually strictly increasing. Hence, $\phi(X)$ is a pseudo-limit of $\phi(E)$ with respect to $v_E$, as claimed.

If $E$ is pseudo-divergent, the same reasoning applies (with the only difference that $\phi_1(E)$ will be pseudo-divergent and $\phi_2(E)$ pseudo-convergent).
\end{proof}

\section{Extensions}\label{sect:fundamentalsatz}
We now start the proof of our generalization of Ostrowski's Fundamentalsatz (Theorem \ref{teor:fundamentalsatz}): we want to show that, under some hypothesis, we can obtain every extension $W$ of $V$ to $K(X)$ as a valuation domain $V_E$ associated to a pseudo-monotone sequence $E$ contained in $K$. In order to accomplish this objective, we want to associate to each such extension $W$ a subset of $K$ which is the analogue of the set of pseudo-limits of a pseudo-monotone sequence.

\begin{Def}
Let $W$ be an extension of $V$ to $K(X)$. We define the following subsets of $K$:
\begin{align*}
\inslim_1(W)= & \{\alpha\in K\mid w(X-\alpha)\notin\Gamma_v\};\\
\inslim_2(W)= & \{\alpha\in K\mid w(X-\alpha)\in\Gamma_v,\text{~and~}w(X-\alpha+c)=w(X-\alpha)\text{~if~}w(X-\alpha)=v(c)\};\\
\inslim(W)= & \inslim_1(W)\cup\inslim_2(W).
\end{align*}
\end{Def}
Equivalently, $\alpha\in\inslim_2(W)$ if $w(X-\alpha)=v(c)$ for some $c\in K$, and the image of $\frac{X-\alpha}{c}$ in the residue field of $W$ does not belong to the residue field of $V$.

\begin{prop}\label{prop:pseudolim}
Let $W$ be an extension of $V$ to $K(X)$.
\begin{enumerate}[(a)]
\item\label{prop:pseudolim:imm} Suppose $K$ is algebraically closed. Then $V\subset W$ is immediate if and only if $\inslim(W)=\emptyset$.
\item\label{prop:pseudolim:w} If $\alpha\in\inslim(W)$, then $w(X-\alpha)\geq w(X-\beta)$ for each $\beta\in K$, and equality occurs if and only if $\beta\in \inslim(W)$.
\item\label{prop:pseudolim:uno} If $\inslim(W)\neq\emptyset$, then exactly one of  $\inslim_1(W)$ and $\inslim_2(W)$ is nonempty.
\item\label{prop:pseudolim:ball1} If $\inslim_1(W)\neq\emptyset$ is nonempty, then it is equal to $K$ or to $\alpha+I$ for some $\alpha\in K$ and some (fractional) ideal $I$.
\item\label{prop:pseudolim:ball2} If $\inslim_2(W)\neq\emptyset$ is nonempty, then it is equal to $\alpha+cV$ for some $\alpha,c\in K$ with $v(c)=w(X-\alpha)$.
\end{enumerate}
\end{prop}
Note that (b) above is a generalization of \cite[Proposition 3.11, (a)]{PS}.
\begin{proof}
\ref{prop:pseudolim:imm} Suppose $K$ is algebraically closed. If $V\subset W$ is immediate, then $\Gamma_w=\Gamma_v$ (so $\inslim_1(W)=\emptyset$); furthermore, since $W/M_W=V/M$, also $\inslim_2(W)=\emptyset$. Conversely, suppose that $V\subset W$ is not immediate. If  $\Gamma_v\neq\Gamma_w$, then $w(p)\notin\Gamma_v$ for some $p\in K[X]$, and thus $w(p')\notin\Gamma_v$ for some irreducible factor $p'$ of $p$; since $K$ is algebraically closed, $p'(X)=X-\alpha$ and $\alpha\in\inslim_1(W)$.
If $\Gamma_v=\Gamma_w$, then $V/M\subsetneq W/M_W$ and this extension must be transcendental (since $K$ is algebraically closed, so is $V/M$). By the proof of \cite[Proposition 2]{AP}, we can find $\alpha,c\in K$ such that $w(X-\alpha)=v(c)$ and the image of $\frac{X-\alpha}{c}$ is transcendental over $V/M$; it follows that $\alpha\in\inslim_2(W)$, which in particular is nonempty.

\ref{prop:pseudolim:w}-\ref{prop:pseudolim:ball2} If $\inslim(W)\neq\emptyset$ and $\inslim(W)\neq K$, let $\alpha\in\inslim(W)$. Then, if $\beta\in K$ we have:
\begin{equation}\label{eq:calcololim}
w(X-\beta)=w(X-\alpha+\alpha-\beta)=\begin{cases}
w(X-\alpha), & \text{if~}v(\alpha-\beta)\geq w(X-\alpha)\\
v(\alpha-\beta), & \text{if~}v(\alpha-\beta)<w(X-\alpha)
\end{cases}
\end{equation}
Suppose $\alpha\in\inslim_1(W)$. Since $w(X-\beta)$ is equal either to $w(X-\alpha)\notin\Gamma_v$ or to  $v(\alpha-\beta)$, in the former case $\beta\in\inslim_1(W)$, while in the latter $\beta\notin\inslim_1(W)$ and $w(X-\beta)<w(X-\alpha)$. Moreover, $\inslim_1(W)=\alpha+\{x\in K\mid v(x)>w(X-\alpha)\}$, and the latter set is an ideal.

If $\alpha\in\inslim_2(W)$ and $v(\alpha-\beta)\geq w(X-\alpha)$, then $w(X-\beta)=w(X-\alpha)=v(c)\in\Gamma_v$, for some $c\in K$, so $\beta\in\inslim_2(W)$ because $(X-\beta)/c=(X-\alpha)/c+(\beta-\alpha)/c$: over the residue field of $W$ $(X-\alpha)/c$ is not in $V/M$ so it follows that the same holds for $(X-\beta)/c$. Similarly, if $\beta\in\inslim_2(W)$ it can be proved that $v(\alpha-\beta)\geq w(X-\alpha)$ and so $w(X-\beta)=w(X-\alpha)$. If $v(\alpha-\beta)<w(X-\alpha)$,  then as before $w(X-\beta)<w(X-\alpha)$.
In particular, $\inslim_2(W)=\alpha+\{x\in K\mid v(x)\geq w(X-\alpha)\}=\alpha+cV$, and $\inslim_1(W)=\emptyset$ (because $\alpha\in\inslim_2(W)$ and \eqref{eq:calcololim}). Note that this argument shows that at most one of the sets $\inslim_i(W)$, $i=1,2$, can be non-empty.

In all cases, $w(X-\alpha)\geq w(X-\beta)$ for all $\alpha\in\inslim(W)$ and $\beta\in K$, and equality occurs if and only if $\beta\in\inslim(W)$.
\end{proof}

\begin{prop}\label{prop:limVE}
Let $E\subset K$ be a pseudo-monotone sequence.
\begin{enumerate}[(a)]
\item\label{prop:limVE:L1} If $E$ is a strictly pseudo-monotone sequence, then $\inslim_1(V_E)=\inslim_E$.
\item\label{prop:limVE:L2} If $E$ is pseudo-stationary, then $\inslim_2(V_E)=\inslim_E$.
\end{enumerate}
In both cases, $\inslim(V_E)$ is the set of pseudo-limits of $E$ in $K$.
\end{prop}
\begin{proof}
Suppose first that $E=\{s_\nu\}_{\nu\in\Lambda}$ is a strictly pseudo-monotone sequence. Let $\alpha\in K$, and suppose $w(X-\alpha)=w(c)$ for some $c\in K$. Then, $(X-\alpha)/c$ is a unit of $W$, and in particular for large $\nu$ both $(s_\nu-\alpha)/c$ and $c/(s_\nu-\alpha)$ belong to $V$. Therefore, $v(s_\nu-\alpha)=v(c)$ for large $\nu$; hence, $w(X-\alpha)\in\Gamma_v$ if and only if $\alpha\notin\inslim_E$. Thus, $\inslim_1(V_E)=\inslim_E$; furthermore, by Proposition \ref{prop:description}\ref{prop:description:alg}, $V_E/M_E=V/M$, and so $\inslim_2(V_E)=\emptyset$. Hence, $\inslim(V_E)=\inslim_E$.

Suppose now that $E$ is pseudo-stationary: then, by Proposition \ref{prop:description}\ref{prop:description:monomiali}, $v_E=v_{\alpha,\delta_E}$. By Proposition \ref{prop:pseudolim}\ref{prop:pseudolim:ball2}, $\mathcal{L}(V_E)=\mathcal{L}_2(V_E)=\alpha+cV$, where $c\in K$ has value $v_E(X-\alpha)=\delta_E$. By Lemma \ref{pseudolimiti} this is precisely $\mathcal{L}_E$.
\end{proof}

\begin{Ex}\label{Vinfty}
Proposition \ref{prop:limVE} allows to show that there are extensions of $V$ to $K(X)$ which cannot be realized as $V_E$, for any pseudo-convergent sequence $E\subset K$. For example, consider the following valuation domain of $K(X)$ introduced in \cite{PerTransc}:
\begin{equation*}
V_{\infty}=\{\phi\in K(X)\mid \phi(\infty)\in V\},
\end{equation*}
where $\phi(\infty)$ is defined as $\psi(0)$, where $\psi(X)=\phi(1/X)$. Then, $V_{\infty}$ is the image of $V_0=\{\phi\in K(X)\mid \phi(0)\in V\}$ under the $K$-automorphism $\Phi$ of $K(X)$ sending $X$ to $1/X$. The valuation domain $V_0$ is equal to $V_F$, where $F=\{t_\nu\}_{\nu\in\Lambda}$ is a Cauchy sequence with limit $0$. Consider $E=\{s_\nu=t_\nu^{-1}\}_{\nu\in\Lambda}$: by Proposition \ref{image pm under a rational function}, $E$ is pseudo-divergent with $\Br(E)=K$ (since, as $v(t_\nu)$ is cofinal in $\Gamma_v$, $v(s_\nu)$ is coinitial). Thus, $V_{\infty}=V_E$ has $\inslim(V_{\infty})=\inslim_E=K$, which is different from $\inslim(V_G)=\inslim_G$ for every pseudo-convergent sequence $G$ (by Lemma \ref{lemma:carattbreadth}). In particular, $V_{\infty}\neq V_G$. Note also that $V_{\infty}$ is contained in the DVR  $K[1/X]_{(1/X)}$ (\cite[Proposition 2.2]{PerTransc}).
\end{Ex}

Proposition \ref{prop:pseudolim}\ref{prop:pseudolim:imm} is false without the assumption on $K$: in fact, if $E\subset K$ is a pseudo-convergent sequence of algebraic type without pseudo-limits in $K$, then, for some extension $u$ of $v$ to $K$, by Proposition \ref{prop:limVE} we have $\mathcal L(U_E)=\mathcal{L}^u_E\not=\emptyset$, so by contracting down to $K$ we have $\mathcal L(V_E)=\mathcal L_E=\emptyset$ while $V\subset V_E$ is not immediate by Proposition \ref{prop:description}.
\begin{prop}\label{prop:uglim}
Suppose $K$ is algebraically closed, and let $W_1,W_2$ be two extensions of $V$ to $K(X)$. If either $\inslim_1(W_1)=\inslim_1(W_2)\neq\emptyset$ or $\inslim_2(W_1)=\inslim_2(W_2)\neq\emptyset$, then $W_1=W_2$.
\end{prop}
\begin{proof}
Let $\inslim=\inslim(W_1)=\inslim(W_2)$; we shall use $w$ to indicate either $w_1$ or $w_2$. Fix also $\alpha\in\inslim$.

Let $\phi\in K(X)$, and write it as $\phi(X)=c\prod_{\gamma\in\Omega}(X-\gamma)^{\epsilon_\gamma}$, where $\Omega$ is the multiset of critical points of $\phi$, $c\in K$ and $\epsilon_\gamma\in\{-1,+1\}$.

For every $\gamma\notin\inslim$, by Proposition \ref{prop:pseudolim}\ref{prop:pseudolim:w} $w(X-\gamma)<w(X-\alpha)$, so $w(X-\gamma)=v(\alpha-\gamma)$; furthermore, if $\gamma_1,\gamma_2\in\inslim$,  then $w(X-\gamma_1)=w(X-\gamma_2)$. Hence, $w(\phi)=w(\psi)$, where $\psi(X)=d(X-\alpha)^t$ for some $d\in K$, $t\inZ$ (more precisely, $d=c\prod_{\gamma\in\Omega\setminus\inslim}(\alpha-\gamma)^{\epsilon_\gamma}$ and $t=\sum_{\gamma\in\Omega\cap\inslim}\epsilon_\gamma$.) Note that, in particular, we have both $w_1(\phi)=w_1(\psi)$ and $w_2(\phi)=w_2(\psi)$.

If $t=0$, then $w(\phi)=v(d)$ and so its sign does not depend on whether $w=w_1$ or $w=w_2$; i.e., $\phi\in W_1$ if and only if $\phi\in W_2$. If $t\neq 0$, then $\psi=(e(X-\alpha)^\epsilon)^{|t|}$, where $e\in K$ is such that $e^{|t|}=d$ and $\epsilon=t/|t|$; thus, $\psi\in W_i$ if and only if $e(X-\alpha)^\epsilon\in W_i$, for $i=1,2$, since a valuation domain is integrally closed.

Suppose now that  $\alpha\in\inslim_1(W)$ and $t>0$. Then,
\begin{equation*}
w(e(X-\alpha))\geq 0\iff w(X-\alpha)\geq v(e^{-1})\iff w(X-\alpha+e^{-1})=v(e^{-1})
\end{equation*}
(since $w(X-\alpha)\notin\Gamma_v$), i.e., if and only if $\alpha-e^{-1}\notin\inslim_1(W)$. Since $\inslim_1(W_1)=\inslim_1(W_2)$, it follows that $w_1(e(X-\alpha))\geq 0$ if and only if $w_2(e(X-\alpha))\geq 0$, i.e., $\phi\in W_1$ if and only if $\phi\in W_2$, as claimed. Analogously, if $t<0$,  then
\begin{equation*}
w\left(\frac{e}{X-\alpha}\right)\geq 0\iff w(X-\alpha)\leq v(e)\iff w(X-\alpha+e)=w(X-\alpha),
\end{equation*}
that is, if and only if $\alpha-e\in\inslim_1(W)$. As before, this implies that $\phi\in W_1$ if and only if $\phi\in W_2$; hence, $W_1=W_2$.

Suppose now that $\alpha\in\inslim_2(W)$. If $t>0$, then $w(e(X-\alpha))\geq 0$ if and only if $w(X-\alpha)>w(f)$ for all $f\in K$ such that $v(e^{-1})>v(f)$; that is, if and only if $w(X-\alpha+f)=v(f)$ for all such $f$. This happens if and only if $\alpha-f\notin\inslim$ for all these $f$; since $v(e^{-1})>v(f)$ depends only on $V$, it follows as before that $w_1(e(X-\alpha))\geq 0$ if and only if $w_2(e(X-\alpha))\geq 0$, i.e., $\phi\in W_1$ if and only if $\phi\in W_2$, as claimed. If $t<0$, then, in the same way, $w(e/(X-\alpha))\geq 0$ if and only if $w(X-\alpha)<v(f)$ for all $f$ such that $v(f)<v(e)$; as above, this implies that $\phi\in W_1$ if and only if $\phi\in W_2$. Hence, $W_1=W_2$.
\end{proof}

\begin{Ex}
In Proposition \ref{prop:uglim} we can't drop the hypothesis that $K$ is algebraically closed: for example, take $\alpha\in K$ and let $\delta\in\insQ\Gamma_v\setminus\Gamma_v$. Let $E\subset K$ be a pseudo-convergent sequence having a pseudo-limit $\alpha$ and such that $\Br(E)=I=\{x\in K\mid v(x)>\delta\}$; by Proposition \ref{prop:limVE}\ref{prop:limVE:L1}, $\inslim_1(V_E)=\alpha+I\neq\emptyset$. Take now the monomial valuation $w=v_{\alpha,\delta}$: then, $\inslim_1(W)=\alpha+I=\inslim_1(V_E)$, but $W\neq V_E$ since the value group of $w$ is contained in the divisible hull of the value group of $v$, while $\Gamma_{v_E}=\Gamma_v\oplus\Delta_E\insZ$ is not (by Proposition \ref{prop:description} and Lemma \ref{minimal degree in PE}).
\end{Ex}

Joining the previous propositions, we can prove that if $K$ is algebraically closed,  then any extension of $V$ to $K(X)$ is in the form $V_E$ for some pseudo-monotone sequence $E$; however, we postpone this result to Theorem \ref{teor:fundamentalsatz} in order to cover a more general case.

\begin{prop}\label{prop:extensionsK(X)}
Let $E\subset K$ be a pseudo-monotone sequence, and let $U$ be an extension of $V$ to $\overline{K}$. Then $U_E$ is the unique common extension of $U$ and $V_E$ to $\overline{K}(X)$. Moreover, if $F\subset K$ is another pseudo-monotone sequence  such that $E$ and $F$ are either both pseudo-stationary or both strictly pseudo-monotone, then $V_E=V_F$ if and only if $U_E=U_F$.
\end{prop}
\begin{proof}
The first claim can be proved in the same way as \cite[Theorem 5.7]{PS}, but we repeat the proof for clarity. Clearly, $U_E$ extends both $U$ and $V_E$. Suppose there is another extension $W$ of $U$ and $V_E$ to $\overline{K}(X)$: then, by \cite[Chapt. VI, \S 8, 6., Corollary 1]{bourbaki-inglese}, there is a $K(X)$-automorphism $\sigma$ of $\overline{K}(X)$ such that $U_E=\sigma(W)$. Let $\rho=\sigma^{-1}$: then,
\begin{equation*}
\begin{aligned}
\rho(U_E)= & \{\rho\circ\phi\in \overline{K}(X)\mid \phi(s_\nu)\in U\text{~eventually}\}=\\
& \{\rho\circ\phi\in \overline{K}(X)\mid \sigma\circ\rho(\phi(s_\nu))\in U\text{~eventually}\}.
\end{aligned}
\end{equation*}
Since $s_\nu\in K$ and $\rho|_K$ is the identity, $\rho(\phi(s_\nu))=(\rho\circ\phi)(s_\nu)$; hence,
\begin{equation*}
\begin{aligned}
\rho(U_E)= & \{\rho\circ\phi\in \overline{K}(X)\mid \sigma((\rho\circ\phi)(s_\nu))\in U\text{~eventually}\}=\\
& \{\psi\in \overline{K}(X)\mid \sigma(\psi(s_\nu))\in U\text{~eventually}\}.
\end{aligned}
\end{equation*}
In particular, note that $\rho(U_E)=\rho(U)_E$.

Since both $U_E$ and $W$ are extensions of $U$, for any $t\in\overline{K}$ we have that $t\in U$ if and only if $\sigma(t)\in U$; in particular, this happens for $t=\psi(s_\nu)$. It follows that $\rho(U_E)=W=U_E$, as claimed. 

We prove now the last claim. One direction is clear, since $V_E=U_E\cap K$ and $V_F=U_F\cap K$. The other implication follows from the previous claim, since $U_E$ is the unique common extension of $V_E$ and $U$ and $U_F$ is the unique common extension of $V_F$ and $U$.
\end{proof}

\section{Equivalence of pseudo-monotone sequences}\label{section:equivalence}
Using the results of the previous sections, we can now tackle the problem of when two pseudo-monotone sequences have the same associated extension of $V$ to $K(X)$.

\begin{prop}\label{prop:equivalence}
Let $E,F\subset K$ be two pseudo-monotone sequences that are either both pseudo-stationary or both strictly pseudo-monotone. Let $u$ be an extension of $v$ to $\overline{K}$. If  $\inslim_E^u\neq\emptyset$, then $V_E=V_F$ if and only if $\inslim_E^u=\inslim_F^u$.
Furthermore, if $\inslim_E\neq\emptyset$, then the previous condition is also equivalent to the corresponding one over $K$.
\end{prop}
\begin{proof}
By Proposition \ref{prop:extensionsK(X)}, it is enough to show that $U_E=U_F$ if and only if $\inslim_E^u=\inslim_F^u$.

Suppose $\inslim_E^u\neq\emptyset$. Then $U\subset U_E$ is not immediate by Proposition \ref{prop:description}, and by Proposition \ref{prop:limVE} $\inslim_E^u=\inslim_2(U_E)$ if $E$ is pseudo-stationary and $\inslim_E^u=\inslim_1(U_E)$ if $E$ is strictly pseudo-monotone. Hence, if $\inslim_E^u=\inslim_F^u$, then also $\inslim^u_F\neq\emptyset$; if $E$ and $F$ are both  pseudo-stationary, then $\inslim_2(U_F)=\inslim_2(U_E)\neq\emptyset$ and so $U_E=U_F$ by Proposition \ref{prop:uglim}, while if $E$ and $F$ are strictly pseudo-monotone the same conclusion holds by the same proposition. Conversely, if $U_E=U_F$, then $\inslim^u_E=\inslim(U_E)=\inslim(U_F)=\inslim^u_F$ and so $E$ and $F$ have the same pseudo-limits (in $\overline{K}$).

Suppose now $\inslim_E\neq\emptyset$. If $\inslim_E^u=\inslim_F^u$ then $\inslim_E=\inslim_F$. Conversely, if $\inslim_E=\inslim_F$, then by Lemma \ref{pseudolimiti} $\Br(E)=\Br(F)$. In particular,  $\Br_u(E)=\Br_u(F)$ so by the same Lemma  $\inslim_E^u=\inslim_F^u$. 
\end{proof}

\begin{oss}\label{equality pcv and pdv}
~\begin{enumerate}
\item Note that, under the same assumptions of Proposition \ref{prop:equivalence}, by Lemma \ref{pseudolimiti}  $E$ and $F$ have the same set of pseudo-limits (either over $K$ or over $\overline K$) if and only if they have the same breadth ideal and they have at least one pseudo-limit in common.
\item It is possible to have $V_E=V_F$ even if $E$ is pseudo-convergent and $F$ is pseudo-divergent: for example, if $I$ is not finitely generated and it is not equal to $cM$ for any $c\in K$, we can find both a pseudo-convergent sequence $E$ and a pseudo-divergent sequence $F$ such that $I=0+I$ is the set of pseudo-limits of $E$ and $F$  (Lemmas \ref{pseudolimiti} and \ref{lemma:carattbreadth}). By Proposition \ref{prop:equivalence}, $V_E=V_F$.
\item If $E,F$ are pseudo-divergent sequences with $\Br(E)=K=\Br(F)$ (that is, if the gauges of $E,F$ are not bounded from below, see \S \ref{pdv}), then $\inslim_E=K=\inslim_F$, and so $V_E=V_F$. This extension is exactly the valuation domain $V_{\infty}$ considered in Example \ref{Vinfty}.
\end{enumerate}
\end{oss}

Let $E,F$ be two Cauchy sequences with limits $x_E,x_F\in K$, respectively. By Proposition \ref{prop:equivalence}, $V_E=V_F$ if and only if $x_E=x_F$; by extending $v$ to the completion $\widehat{K}$, we see that this can happen even if the limits are not in $K$. Thus, the condition $V_E=V_F$ generalizes the notion of equivalence between Cauchy sequences: for this reason, we say that two pseudo-monotone sequences are \emph{equivalent} if $V_E=V_F$. We now want to characterize this notion in a more intrinsic way, but we need to distinguish between the different types. The first result, involving pseudo-convergent sequences, is a generalization of \cite[Theorem 5.4]{PS}.

\begin{prop}\label{prop:equivalence-pcv}
Let $E=\{s_\nu\}_{\nu\in\Lambda},F=\{t_\mu\}_{\mu\in\Lambda}\subset K$ be pseudo-convergent sequences. Then $E$ and $F$ are equivalent if and only if $\Br(E)=\Br(F)$ and, for every $\kappa\in\Lambda$, there are $\nu_0,\mu_0\in\Lambda$ such that, whenever $\nu\geq \nu_0$, $\mu\geq \mu_0$, we have
$v(s_\nu-t_\mu)>v(t_{\rho}-t_\kappa)$, for any $\rho>\kappa$.
\end{prop}
Note that the condition of the proposition is not symmetrical in $E$ and $F$, despite the fact that the definition of the equivalence relation is symmetric.
\begin{proof}
By Proposition \ref{prop:equivalence}, without loss of generality we can suppose that $K$ is algebraically closed. Let $\{\delta_\nu\}_{\nu\in\Lambda},\{\delta'_\nu\}_{\nu\in\Lambda}$  be the gauges of $E$ and $F$, respectively. We will use the following remark:  $\Br(E)\subseteq \Br(F)$ if and only if for each $\mu\in\Lambda$ there exists $\nu\in\Lambda$ such that $\delta'_\mu\leq\delta_\nu$.

We assume first that the conditions of the statement hold. Suppose that $E$ is of algebraic type: then, $E$ has a pseudo-limit $\beta\in K$. Fix $\mu\in\Lambda$. By the above remark, there exists $\nu_0\in\Lambda$ such that for all $\nu\geq \nu_0$, $\delta_\nu>\delta'_\mu$.  There also exist $\iota_0,\mu_0\in\Lambda$ such that for all $\nu\geq \iota_0$, $\kappa\geq \mu_0$, we have $v(s_\nu-t_\kappa)>\delta'_\mu$.  Then, for $\nu\geq\max\{\iota_0,\nu_0\}$ and $\kappa>\max\{\mu,\mu_0\}$ we have
\begin{equation*}
v(\beta-t_\mu)=v(\beta-s_\nu+s_\nu-t_\kappa+t_\kappa-t_\mu)=\delta'_\mu
\end{equation*}
so that $\beta$ is a pseudo-limit of $F$. Therefore, $F$ is of algebraic type and $\mathcal{L}_E\subseteq \mathcal{L}_F$. The reverse inclusion is proved symmetrically, and $V_E=V_F$ follows from Proposition \ref{prop:equivalence}.

Suppose now that $E$ is of transcendental type: by the previous part of the proof, also $F$ must be of transcendental type. We can repeat the previous reasoning by using $X$ instead of $\beta$ (since $X$ is a pseudo-limit of $E$ with respect to $v_E$: see \cite[Theorem 3.8]{PS} or Theorem \ref{VE valuation domain}); this proves that $X$ is a pseudo-limit of $F$ with respect to $v_F$. The fact that $V_E=V_F$ now follows from \cite[Theorem 2]{kaplansky-maxfield}.

Assume now that $V_E=V_F$. Suppose first that $E$ is of algebraic type: then, $\inslim_E\neq\emptyset$, and by Proposition \ref{prop:equivalence} we must have $\inslim_F=\inslim_E$, and thus $F$ is also of algebraic type. In particular, $\Br(E)=\Br(F)$. Let $\alpha\in\inslim_E=\inslim_F$. Then,
\begin{equation*}
v(s_\nu-t_\mu)=v(s_\nu-\alpha+\alpha-t_\mu)\geq\min\{\delta_\nu,\delta'_\mu\}.
\end{equation*}
By the remark, for every $\kappa$ there is an $\iota_0$ such that $\delta_{\iota_0}>\delta'_\kappa$; choosing $\mu_0>\kappa$ we have that $E$ and $F$ satisfy the conditions of the statement.

Suppose now that $E$ is of transcendental type; as before, this implies that also $F$ is of transcendental type. Without loss of generality we may suppose that $\Br(F)\subseteq\Br(E)$. If this containment is strict, then there exists a $c\in\Br(E)\setminus\Br(F)$. Then, $\frac{c}{X-\alpha}$ is in $V_E$ for each $\alpha\in K$ (because $X$ is a pseudo-limit of $E$ with respect to $v_E$ and $E$ has no pseudo-limits in $K$). On the other hand, for every $\nu$ we have $\frac{c}{X-t_\nu}\notin V_F$, a contradiction. Therefore $\Br(E)=\Br(F)$.  We know that $X$ is a pseudo-limit of $F$ with respect to $v_F$, so that $\{v_F(X-t_\mu)\}_{\mu\in\Lambda}$ is a (eventually) strictly increasing sequence. In particular, since $V_E=V_F$ implies that $\lambda\circ v_E=v_F$ for some isomorphism of totally ordered groups $\lambda:\Gamma_{v_E}\to\Gamma_{v_F}$, it follows that $\{v_E(X-t_\mu)\}_{\mu\in\Lambda}$ is a (eventually) strictly increasing sequence, so that $X$ is a pseudo-limit of $F$ with respect to $v_E$. Thus $v_E(X-t_\mu)=\delta'_\mu$, for each $\mu\in\Lambda$ (sufficiently large). The proof now proceeds as above, replacing a pseudo-limit $\alpha$ of $E$ and $F$ by $X$ (which is a pseudo-limit of $E$ and $F$ with respect to $v_E$). Hence, the conditions of the statement holds.
\end{proof}

The cases of pseudo-divergent and pseudo-stationary sequences are very similar, with the further simplification that in these cases we do not need to consider sequences of transcendental type (which do not exist).
\begin{prop}\label{prop:equivalence-pdv}
Let $E=\{s_\nu\}_{\nu\in\Lambda},F=\{t_\mu\}_{\mu\in\Lambda}\subset K$ be pseudo-divergent sequences. Then $E$ and $F$ are equivalent if and only if $\Br(E)=\Br(F)$ and there exist $\nu_0, \mu_0\in\Lambda$ such that for all $\nu\geq \nu_0, \mu\geq \mu_0$ there exists $\kappa\in\Lambda$ such that $v(s_\nu-t_\mu)\geq  v(t_{\rho}-t_\kappa)$, for any $\rho<\kappa$.
\end{prop}
Note that the above condition amounts to saying that $s_\nu-t_\mu$ is eventually in the breadth ideal $\Br(E)=\Br(F)$.

The following is the analogous result for pseudo-stationary sequences.
\begin{prop}\label{prop:equivalence-pstaz}
Let $E=\{s_\nu\}_{\nu\in\Lambda},F=\{t_\mu\}_{\mu\in\Lambda}\subset K$ be pseudo-stationary sequences with breadth $\delta_E$ and $\delta_F$, respectively. Then $E$ and $F$ are equivalent if and only if $\delta_E=\delta_F=\delta$ and $v(s_\nu-t_\mu)\geq\delta$ for all $\nu,\mu\in\Lambda$.
\end{prop}
\begin{proof}
The conditions of the statement say (using Lemma \ref{pseudolimiti}) that $E\subset\inslim_F$ and $F\subset\inslim_E$. By the same Lemma, this is equivalent to $\inslim_E=\inslim_F$, which is equivalent to $V_E=V_F$ by Proposition \ref{prop:equivalence}.
\end{proof}

\section{A generalized Fundamentalsatz}\label{generalized Fundamentalsatz}

In general, not all the extensions of $V$ to $K(X)$ can be realized via a pseudo-monotone sequence contained in $K$. For example, let $V$ be the ring of $p$-adic integers $\Z_p$, for some prime $p\in\Z$. It is not difficult to see that for $\alpha\in \overline{\Q_p}\setminus\Q_p$, the valuation domain $V_{p,\alpha}=\{\phi\in \Q_p(X)\mid \phi(\alpha)\in \overline{\Z_p}\}$ of $\Q_p(X)$, where $\overline{\Z_p}$ is the unique valuation domain of $\overline{\Q_p}$, is not of the form $V_E$, for any pseudo-monotone sequence $E\subset \Q_p$ (for example, by Proposition \ref{prop:description} and \cite[Proposition 2.2 \& Theorem 3.2]{PerTransc}, see also the proof of Theorem \ref{teor:fundamentalsatz}).

In this section, we show when all extensions of $V$ to $K(X)$ are induced by  pseudo-monotone sequences in $K$. We start with a lemma which allows us to reduce to the algebraically closed case.
\begin{lemma}\label{lemma:pmon-dentro}
Let $L$ be an extension of $K$ and $U$  a valuation domain of $L$ lying over $V$ such  that $\Gamma_u=\Gamma_v$. Let $F\subset L$ be a pseudo-monotone sequence with respect to $u$ having a pseudo-limit $\beta\in K$. Then:
\begin{enumerate}[(a)]
\item\label{lemma:pmon-dentro:strictly} if $F$ is strictly pseudo-monotone, there is a sequence $E\subset K$ of the same kind as $F$ that is equivalent to $F$ (with respect to $u$);
\item\label{lemma:pmon-dentro:pstaz} if $F$ is pseudo-stationary and the residue field of $V$ is infinite, there is a pseudo-stationary sequence $E\subset K$ that is equivalent to $F$ (with respect to $u$).
\end{enumerate}
\end{lemma}
\begin{proof}
Let $F=\{t_\nu\}_{\nu\in\Lambda}$. 

\ref{lemma:pmon-dentro:strictly} For every $\nu$, there is a $c_\nu\in K$ such that $u(t_\nu-\beta)=u(c_\nu)=\delta_\nu$; let $s_\nu=c_\nu+\beta$ and let $E=\{s_\nu\}_{\nu\in\Lambda}$. Then, $E\subset K$ (since $\beta\in K$) and
\begin{equation*}
u(s_\mu-s_\nu)=u(c_\mu+\beta-c_\nu-\beta)=u(c_\mu-c_\nu)=\delta_\nu
\end{equation*}
for every $\mu>\nu$, so $E$ is pseudo-monotone of the same kind as $F$ and the gauges of $E$ and $F$ coincide; in particular, $\Br_u(E)=\Br_u(F)$. By Proposition \ref{prop:equivalence}, $E$ and $F$ are equivalent.

\ref{lemma:pmon-dentro:pstaz} Since $u(t_\nu-\beta)=\delta\in\Gamma_v$ and the residue field of $V$ is infinite, we can find an infinite set $\{c_\nu\}_{\nu\in\Lambda}\subset V$ such that $u(c_\nu-\beta)=\delta$ and such that $u(c_\nu-c_\mu)=\delta$ for every $\nu\neq \mu$. Setting $s_\nu=c_\nu+\beta$, as in the previous case we can take $E=\{s_\nu\}_{\nu\in\Lambda}$, and $E$ and $F$ are equivalent by Proposition \ref{prop:equivalence}.
\end{proof}

\begin{teor}\label{teor:fundamentalsatz}
Let $V$ be a valuation domain with quotient field $K$. Then, every extension $W$ of $V$ to $K(X)$ is of the form $W=V_E$ for some pseudo-monotone sequence $E\subset K$ if and only if $\widehat K$ is algebraically closed. In this case, we have the following.
\begin{enumerate}[(a)]
\item If $V\subset W$ is immediate, then $E$ is necessarily a pseudo-convergent sequence of transcendental type.
\item If $V\subset W$ is not immediate, then:
\begin{enumerate}[label=(\alph{enumi}\arabic*)]
\item\label{teor:fundamentalsatz:empty} if $\inslim(W)=\emptyset$, then $E$ is a pseudo-convergent Cauchy sequence of algebraic type whose limit is in $\widehat{K}\setminus K$;
\item\label{teor:fundamentalsatz:div} if $\mathcal{L}_1(W)=\alpha+I\neq\emptyset$ and $I$ is a divisorial fractional ideal, then $E$ can be taken to be pseudo-convergent of algebraic type;
\item\label{teor:fundamentalsatz:nonprinc} if $\mathcal{L}_1(W)=\alpha+I\neq\emptyset$ and $I$ is not a principal ideal, then $E$ can be taken to be pseudo-divergent;
\item\label{teor:fundamentalsatz:pstaz} if $\mathcal{L}_2(W)=\alpha+I\neq\emptyset$, then $E$ is necessarily a pseudo-stationary sequence.
\end{enumerate}
\end{enumerate}
\end{teor}
Note that, since every nondivisorial ideal is nonprincipal, cases \ref{teor:fundamentalsatz:div} and \ref{teor:fundamentalsatz:nonprinc} cover all possibilities. Furthermore, these two cases are not mutually exclusive: see Remark \ref{equality pcv and pdv}.
\begin{proof}
Throughout the proof we will use the fact that $\widehat K$ is algebraically closed if and only if $\overline K$ embeds in $\widehat K$ (which in turn follows from the fact that the completion of an algebraically closed field is algebraically closed \cite[\S 15.3, Theorem 1]{shell}). Loosely speaking, this condition holds if and only if $K$ is dense in its algebraic closure $\overline{K}$.

Suppose that $\widehat{K}$ is not algebraically closed. Then by above there exists $\alpha\in\overline{K}$ such that $K(\alpha)$ cannot be embedded into $\widehat{K}$, that is, $\alpha$ is not the limit of any Cauchy sequence in $K$. Let $U$ be an extension of $V$ to $\overline{K}$ and let $F\subset\overline{K}$ be a pseudo-convergent Cauchy sequence with limit $\alpha$. Let $W=U_F\cap K(X)$: we claim that $W\neq V_E$ for any pseudo-monotone sequence $E$. Indeed, if $W=V_E$ for some pseudo-monotone sequence $E\subset K$, by Proposition \ref{prop:extensionsK(X)} $U_E$ is the only common extension of $U$ and $V_E$ to $\overline{K}(X)$, so that $U_E=U_F$. By Proposition \ref{prop:equivalence}, we must have $\inslim_E^u=\inslim_F^u=\{\alpha\}$ and $\Br_u(E)=\Br_u(F)=(0)$ and thus $\inslim_E=\inslim_E^u\cap K=\emptyset$; hence, $E\subset K$ should be a pseudo-convergent Cauchy  sequence with limit $\alpha$ (Lemma \ref{pseudolimiti}).  However, this is impossible by the choice of $\alpha$, and so $W\neq V_E$ for any pseudo-monotone sequence $E$. 

\medskip

Suppose now that $\widehat{K}$ is algebraically closed, and let $\mathcal{W}$ be a common extension of $\widehat{V}$ and $W$ to $\widehat{K}(X)$.

If $\widehat{V}\subset \mathcal{W}$ is immediate, then also $V\subset W$ is immediate (since $V\subset\widehat{V}$ is); by Kaplansky's Theorem  \cite[Theorem 2]{kaplansky-maxfield}, there is a pseudo-convergent sequence $E\subset K$ such that $W=V_E$.

Suppose $\widehat{V}\subset\mathcal{W}$ is not immediate. By Proposition \ref{prop:pseudolim}\ref{prop:pseudolim:imm}, $\inslim(\mathcal{W})\subseteq\widehat{K}$ is nonempty, say equal to $\alpha+J$ for some $\alpha\in\widehat{K}$ and some $J$ that is either a fractional ideal of $\widehat{V}$ or the whole $\widehat{K}$.

If $J=(0)$ let $E\subset K$ be a pseudo-convergent Cauchy sequence having limit $\alpha$: then, $\inslim(\widehat{V}_E)=\inslim_1(\widehat{V}_E)=\{\alpha\}=\inslim_1(\mathcal{W})=\inslim(\mathcal{W})$, and by Proposition \ref{prop:uglim} it follows that $\mathcal{W}=\widehat{V}_E$. Hence, $W=\mathcal{W}\cap K=\widehat{V}_E\cap K=V_E$. In particular, if $\alpha\in K$ then $\inslim(W)=\{\alpha\}$, while if $\alpha\in\widehat{K}\setminus K$ then $\inslim(W)=\emptyset$; furthermore, by Proposition \ref{prop:description} if $V\subset W$ is not immediate, then $E$ must be a sequence of algebraic type. 

Suppose now that $J\neq(0)$. Then, the open set $\alpha+J$ must contain an element $\beta$ of $K$, and in particular $\alpha+J=\beta+J$. Using Lemma \ref{lemma:carattbreadth}, we construct a pseudo-monotone sequence $F\subset \widehat K$ with breadth ideal $J$ and with $\beta$ as pseudo-limit, with the following properties:
\begin{itemize}
\item if $\inslim_1(\mathcal{W})\neq\emptyset$ and $J$ is a strictly divisorial fractional ideal, we take $F$ to be a pseudo-convergent sequence;
\item if $\inslim_1(\mathcal{W})\neq\emptyset$ and $J$ is a nondivisorial fractional ideal, we take $F$ to be a pseudo-divergent sequence (note that, in this case, $J=c\widehat{M}$ is not principal);
\item if $\inslim_1(\mathcal{W})=\widehat{K}$, we take $F$ to be a pseudo-divergent sequence whose gauge is coinitial in $\Gamma_v$;
\item if $\inslim_2(\mathcal{W})\neq\emptyset$, we take $F$ to be a pseudo-stationary sequence.
\end{itemize}
Note that the first case falls in \ref{teor:fundamentalsatz:div}, the second and the third ones in \ref{teor:fundamentalsatz:nonprinc} and the fourth one in \ref{teor:fundamentalsatz:pstaz}.

In all cases, $\mathcal{W}=\widehat{V}_F$ by Proposition \ref{prop:uglim} (in the first three cases using $\inslim_1$ and in the last one using $\inslim_2$). Since $V\subset\widehat{V}$ is immediate, we can apply Lemma \ref{lemma:pmon-dentro} to find a pseudo-monotone sequence $E\subset K$ that is equivalent to $F$; hence, $V_E=\widehat{V}_F\cap K=\mathcal{W}\cap K=W$. The theorem is now proved.
\end{proof}

\begin{oss}\label{oss su Ostrowski}
By Proposition \ref{prop:description} and the main Theorem \ref{teor:fundamentalsatz}, if $\widehat K$ is algebraically closed, then every extension of $V$ to $K(X)$ which is not immediate is a monomial valuation. This result was already known to hold but only with the stronger assumption that $K$ is algebraically closed, see \cite[pp. 286-289]{PopAllvaluations}. 

We remark that a more direct approach to the proof of Theorem \ref{teor:fundamentalsatz} can be given by considering the set $w(X,K)=\{w(X-a) \mid a\in K\}$, which is a subset of $\Gamma_w$. If $w(X,K)$ has no maximum, then, exactly as in the original proof of Ostrowski, we can extract from $w(X,K)$ a cofinal sequence which determines a pseudo-convergent sequence $E$ in $K$ of transcendental type such that $W=V_E$. If instead $w(X,K)$ has a maximum $\Delta_w=w(X-a_0)$, then, following again Ostrowski's proof, one can show that $W$ is a monomial valuation of the form $V_{a_0,\Delta_w}$: according to whether $\Delta_w$ is in $\Gamma_v$ or not (and, in the latter case, depending on the properties of the cut induced by $\Delta_w$ on $\Gamma_v$), we can find a pseudo-monotone sequence $E\subset K$ with $a_0$ as pseudo-limit and such that $W=V_E$. This approach can be connected to the one given above by noting that $\Gamma_v\setminus w(X,K)=v(J)$ (where $J$ is the ideal defined in the proof of Theorem \ref{teor:fundamentalsatz}), and that if $\Delta_w$ exists then we have $\{a\in K \mid w(X-a)=\Delta_w\}=\mathcal{L}(W)$.

When $w(X,K)$ has a maximum and $V$ has rank 1, Ostrowski proved in his Fundamentalsatz \cite[p. 379]{Ostrowski} that the rank one valuation associated to $W$ can be realized through a pseudo-convergent sequence $F=\{s_\nu\}_{\nu\in\Lambda}\subset K$ by means of the map defined as $w_F(\phi)=\lim_{\nu\to\infty} v(\phi(s_\nu))$, for each $\phi\in K(X)$ (where the limit is taken in $\insR$). If $W=V_E$, where $E\subset L$ is a pseudo-stationary sequence, as in Theorem \ref{teor:fundamentalsatz}\ref{teor:fundamentalsatz:pstaz}, then $E$ and $F$ have the same set of pseudo-limits, and in particular they have the same breadth. Furthermore, by Proposition \ref{prop:subext-VF} below, in this case we have $V_F\subset W=V_E$. See also \cite{PS} for other results regarding the valuation $w_F$ introduced by Ostrowski. 
\end{oss}

An immediate corollary of Theorem \ref{teor:fundamentalsatz} is that, for any field $K$, if $U$ is an extension of $V$ to $\overline{K}$, then every extension $W$ of $V$ to $K(X)$ can be written as the contraction of $U_E$ to $K(X)$, namely $U\cap K(X)$, where $E\subset\overline{K}$ is a pseudo-monotone sequence with respect to $U$; furthermore, in view of the examples above, we cannot always choose $E$ to be contained in $K$.

\begin{oss}
The hypothesis that $\widehat{K}$ is algebraically closed is weaker than the hypothesis that $K$ is algebraically closed; we give a few examples.
~\begin{enumerate}
\item  Let $K=\overline{\insQ}\cap\insR$, where $\overline{\insQ}$ is the algebraic closure of $\insQ$, and let $V$ be an extension to $K$ of $\insZ_{(5\insZ)}$. Then, $i$ belongs to the completion $\widehat{K}$, since the polynomial $X^2+1$ has a root in $\insZ/5\insZ$; therefore, $K(i)=\overline{\insQ}$ can be embedded into $\widehat{K}$, so $\widehat{K}$ is algebraically closed while $K$ is not. 

\item If $K$ is separably closed, then $\widehat{K}$ is algebraically closed (it is enough to adapt the proof of \cite[Chapter 2, (N)]{ribenboim} to the general case).

\item Suppose that $V$ has rank $1$ and that the residue field $k$ has characteristic $0$. Then, $\widehat{K}$ is algebraically closed if and only if $k$ is algebraically closed and the value group $\Gamma_v$ is divisible. Indeed, these two conditions are necessary, since completion preserves value group and residue field. Conversely, suppose that the two conditions hold. When $V$ has rank 1 then $\widehat{K}$ is henselian, i.e. $\widehat{v}$ has a unique extension to the algebraic closure of $\widehat{K}$. Since $k$ has characteristic $0$, all finite extensions of $\overline{K}$ are defectless \cite[Corollary 20.23]{endler}, and thus the fundamental inequality is an equality and thus the degree $[\overline{\widehat{K}}:\widehat{K}]=1$, i.e., $\widehat{K}$ is algebraically closed.
\end{enumerate}
\end{oss}

\section{Geometrical interpretation}\label{geometrical}

Throughout this section, we suppose that the maximal ideal $M$ of $V$ is not finitely generated, and that its residue field $k$ is infinite. We also fix  $\alpha\in K$. For any $\delta\in\Gamma_v$, we denote $B(\alpha,\delta)=\{x\in K\mid v(\alpha-x)\geq\delta\}$ and $\mathring{B}(\alpha,\delta)=\{x\in K\mid v(\alpha-x)>\delta\}$ the closed and open ball (respectively) of center $\alpha$ and radius $\delta$.

By Lemma \ref{lemma:carattbreadth}, we can find both a pseudo-convergent sequence $E$ and a pseudo-stationary sequence $F$ such that $\inslim_E=\inslim_F=\alpha+cV=B(\alpha,\delta)$, where $\delta=v(c)$, for some $c\in K$; furthermore, again by Lemma \ref{lemma:carattbreadth}, for every $z\in\alpha+cV$ we can find a pseudo-divergent sequence $D_z$ such that $\inslim_{D_z}=z+cM=\mathring{B}(z,\delta)$. Note that by Lemma \ref{pseudolimiti}  $D_z\subset \mathring{B}(z,\delta)$. In geometrical terms, $E$ and $F$ are associated to the closed ball $B(\alpha,\delta)$, while each $D_z$ is associated to the open ball $\mathring{B}(z,\delta)$, which is contained in $B(\alpha,\delta)$ and has the same radius. In the next proposition we show the containments among the valuation domains associated to these sequences.
\begin{prop}\label{prop:subext-VF}
Preserve the notation above, and let $z,w\in B(\alpha,\delta)$. Then, $V_F$ properly contains both $V_E$ and $V_{D_z}$, $V_E\neq V_{D_z}$, and $V_{D_z}=V_{D_w}$ if and only if $z-w\in cM$.
\end{prop}
\begin{proof}
Let $U$ be an extension of $V$ to $\overline{K}$ and let $z\in B(\alpha,\delta)$: then
$$\inslim^u_{D_z}=z+cM_U\subsetneq\inslim^u_E=\inslim_F^u=\alpha+cU$$
Let $\phi\in K(X)$ and, for any sequence $G$, let $\lambda_G$ be the dominating degree of $\phi$ with respect to $G$.

Since $\inslim^u_E=\inslim_F^u$, we have $\lambda_E=\lambda_F$; if $E=\{s_\nu\}_{\nu\in\Lambda}$ and $\{\delta_\nu\}_{\nu\in\Lambda}$ is the gauge of $E$, by Proposition \ref{values rat func on psm seq} for large $\nu$ we have $v(\phi(s_\nu))=\lambda_E\delta_\nu+\gamma$, where $\gamma=u\left(\frac{\phi}{\phi_S}(\alpha)\right)$. If $\phi\in V_E$, then $v(\phi(s_\nu))\geq 0$ for all $\nu$ sufficiently large; since $\delta_\nu\nearrow\delta$, it follows that $\lambda_E\delta+\gamma\geq 0$. However, if $F=\{t_\nu\}_{\nu\in\Lambda}$, then applying again Proposition \ref{values rat func on psm seq} we have $v(\phi(t_\nu))=\lambda_F\delta+\gamma=\lambda_E\delta+\gamma$, where $\gamma$ is the same as the previous case; it follows that $v(\phi(t_\nu))\geq 0$ for large $\nu$, i.e., $\phi\in V_F$. 

Fix now $z\in B(\alpha,\delta)$ and let $D_z=\{r_\nu\}_{\nu\in\Lambda}$. Let $\{\delta_\nu'\}_{\nu\in\Lambda}$ be the gauge of $D_z$; by mimicking the proof of Proposition \ref{values rat func on psm seq}, we have
\begin{equation*}
\phi(X)=d\prod_{\alpha\in \inslim^u_{D_z}\cap S}(X-\alpha)^{\epsilon_{\alpha}}\prod_{\beta\in (\inslim^u_F\setminus\inslim^u_{D_z})\cap S}(X-\beta)^{\epsilon_{\beta}}\prod_{\gamma\notin \inslim^u_F\cap S}(X-\gamma)^{\epsilon_{\gamma}},
\end{equation*}
for some $d\in K$, where $S$ is the multiset of critical points of $\phi$. Hence, for large $\nu$, $v(\phi(r_\nu))=\lambda_{D_z}\delta_\nu'+(\lambda_F-\lambda_{D_z})\delta+\gamma$. As in the previous case, if $\phi\in V_{D_z}$ then $v(\phi(r_\nu))\geq 0$ for large $\nu$, and so $0\leq \lambda_F\delta+\gamma=v(\phi(t_\nu))$, i.e., $\phi\in V_F$.

Thus $V_E$ and the $V_{D_z}$ are contained in $V_F$; the containment is strict by Proposition \ref{prop:description}, since $V_F$ is residually transcendental over $V$ while the others are not. The last two claims follow from Lemma \ref{pseudolimiti} and Proposition \ref{prop:equivalence} by comparing the set of the pseudo-limits of the sequences involved.
\end{proof}

Consider now the quotient map $\pi:V_F\longrightarrow V_F/M_F$. By Proposition \ref{prop:description}\ref{prop:description:pstaz}, $V_F/M_F\simeq k(t)$, where $t$ is the image of $\frac{X-\alpha}{c}$. Let $W$ be either $V_E$ or $V_{D_z}$ for some $z\in B(\alpha,\delta)$: then, $M_F\subset W$, and thus we can consider the quotient $\pi(W)=W/M_F$, obtaining the following commutative diagram:
\begin{equation}\label{eq:subext}
\xymatrix{
V\ar@{^{(}->}[r]\ar@{>>}[d]_{\pi} & W\ar@{^{(}->}[r]\ar@{>>}[d]_{\pi} & V_F\ar@{>>}[d]_{\pi}\\
V/M=k\ar@{^{(}->}[r]& W/M_{F}\ar@{^{(}->}[r]& k(t).
}
\end{equation}
In particular, $W/M_F$ is a (proper) valuation domain of $k(t)$ containing $k$: hence, by \cite[Chapter 1, \S 3]{Chevalley}, $W/M_F$ must be equal either to $k[t]_{(f(t))}$, for some irreducible polynomial $f\in k[t]$, or to $k[1/t]_{(1/t)}$. In particular, $\pi$ induces a one-to-one correspondence between the valuation domains of $k(t)$ containing $k$ and the valuation domains of $K(X)$ contained in $V_F$. The strictly pseudo-monotone sequences we considered above are exactly the linear case, as we show next.
\begin{prop}\label{prop:linearquoz}
Preserve the notation above. Then:
\begin{enumerate}[(a)]
\item $\pi(V_E)=k[1/t]_{(1/t)}$;
\item $\pi(V_{D_z})=k[t]_{(t-\theta(z))}$, where $\theta(z)=\pi\left(\frac{z-\alpha}{c}\right)$;
\item $\pi^{-1}(k[t]_{(t-x)})=V_{D_{\alpha+yc}}$, where $y$ is an element of $V$ satisfying $\pi(y)=x$.
\end{enumerate}
\end{prop}
\begin{proof}
Let $\phi(X)=\frac{X-\alpha}{c}$: then, as in the previous discussion, $t=\pi(\phi)$. The ring $k[1/t]_{(1/t)}$ is the only valuation domain of $k(t)$ containing $k$ such that $1/t$ belongs to the maximal ideal: hence, in order to show that $\pi(V_E)=k[1/t]_{(1/t)}$ we only need to show that $1/\phi\in M_E$. This follows immediately from the fact that $v(\phi(s_\nu))=\delta_\nu-\delta<0$, where $E=\{s_\nu\}_{\nu\in\Lambda}$ and $\{\delta_\nu\}_{\nu\in\Lambda}$ is the gauge of $E$.

Analogously, in order to show that $\pi(V_{D_z})=k[t]_{(t-\theta(z))}$, we need to show that $t-\theta(z)$ is in the maximal ideal of $\pi(V_{D_z})$ or, equivalently, that
\begin{equation*}
\phi(X)-\frac{z-\alpha}{c}=\frac{X-z}{c}\in M_{D_z}.
\end{equation*}
This is an immediate consequence of the definition of $z$ and $c$, and the claim is proved.

The last point follows by the fact that $\theta(\alpha+yc)=\pi\left(\frac{\alpha+yc-\alpha}{c}\right)=\pi(y)=x$.
\end{proof}

If $k$ is algebraically closed (in particular, if $K$ is algebraically closed), then all irreducible polynomials of $k[t]$ are linear; thus, Proposition \ref{prop:linearquoz} describes all the subextensions of $V_F$. When $k$ is not algebraically closed, on the other hand, it follows that some of the valuation rings of $k(t)$ containing $k$  cannot be obtained by pseudo-divergent sequences contained in $K$ in the same way as in Proposition \ref{prop:linearquoz}; however, we can construct them by using pseudo-divergent sequences in $\overline{K}$ with respect to a fixed extension of $V$.

Given an extension $u$ of $v$ to $\overline{K}$ we denote by $\mathcal D(U)$ the \emph{decomposition group} of $U$ in $\Gal(\overline{K}/K)$, that is, $\mathcal D(U)=\{\sigma\in \Gal(\overline{K}/K) \mid \sigma(U)=U\}$.
\begin{prop}\label{prop:subext-nonac}
Let $W$ be an extension of $V$ to $K(X)$ which is properly contained in $V_F$, and suppose that $\pi(W)=k[t]_{(f(t))}$ for some nonlinear irreducible $f\in k[t]$. Let $u$ be an extension of $v$ to $\overline{K}$.
\begin{enumerate}[(a)]
\item There exists $z\in\mathcal{L}_F^u$ such that $W=U_{D_z}\cap K(X)$, where $D_z\subset \mathring B_u(z,\delta)\subset \overline K$ is pseudo-divergent.
\end{enumerate}
Let $\overline{\pi}:U_F\to \overline{k}(t)$ be the canonical residue map.
\begin{enumerate}[(a),resume]
\item\label{prop:subext-nonac:ovtheta} $\overline{\theta}(z)=\overline{\pi}\left(\frac{z-\alpha}{c}\right)$ is a zero of $f(t)$.
\item\label{prop:subext-nonac:equiv} Let $z,w\in\mathcal L_F^u$. Then the following are equivalent:
\begin{enumerate}[(i)]
\item\label{prop:subext-nonac:equiv:Udz} $U_{D_z}\cap K(X)=U_{D_w}\cap K(X)$;
\item\label{prop:subext-nonac:equiv:ovtheta} $\overline{\theta}(z)$ and $\overline{\theta}(w)$ are conjugate over $k$;
\item\label{prop:subext-nonac:equiv:decomp} $\rho(z)-w\in cM_U$ for some $\rho\in\mathcal D(U)$.
\end{enumerate}
In particular, the number of extensions of $W$ to $\overline{K}(X)$ is equal to the number of distinct roots of $f$ in $\overline{k}$.
\end{enumerate}
\end{prop}

\begin{proof}
Let $\mathcal{W}$ be an extension of $W$ to $\overline{K}(X)$ and let $U=\mathcal{W}\cap\overline{K}$; then, $U$ is an extension of $V$. The diagram \eqref{eq:subext} lifts to
\begin{equation}\label{eq:subext-ac}
\xymatrix{
U\ar@{^{(}->}[r]\ar@{>>}[d]_{\overline{\pi}} & \mathcal{W}\ar@{^{(}->}[r]\ar@{>>}[d]_{\overline{\pi}} & U_F\ar@{>>}[d]_{\overline{\pi}}\\
U/M_U=\overline{k}\ar@{^{(}->}[r]& \mathcal{W}/M_U\ar@{^{(}->}[r]& \overline{k}(t).
}
\end{equation}
By Proposition \ref{prop:linearquoz}, $\mathcal{W}$ is equal to  $U_{D_z}$, for some $z\in\inslim^u_F=\alpha+cU$, and thus $W=U_{D_z}\cap K(X)$, as desired.

\ref{prop:subext-nonac:ovtheta} If $U_{D_z}$ is an extension of $W$ to $\overline{K}(X)$, then $\overline{\pi}(U_{D_z})=\overline{k}[t]_{(t-\overline{\theta}(z))}$ is an extension of $\pi(W)=k[t]_{(f(t))}$ to $\overline{k}(t)$. It is straightforward to see that this implies that $t-\overline{\theta}(z)$ is a factor of $f(t)$ in $\overline{k}[t]$, i.e., that $\overline{\theta}(z)$ is a zero of $f(t)$.

\ref{prop:subext-nonac:equiv} The equivalence of \ref{prop:subext-nonac:equiv:Udz} and \ref{prop:subext-nonac:equiv:ovtheta} follows from the previous point.

\ref{prop:subext-nonac:equiv:ovtheta} $\iff$ \ref{prop:subext-nonac:equiv:decomp} There is a surjective map from the decomposition group $\mathcal{D}(U)$ of $U$ to the Galois group $\Gal(\overline{k}/k)$, where $\rho\in\mathcal{D}(U)$ goes to the map $\overline{\rho}$ sending $x\in k$ to $\pi(\rho(y))$, where $y$ satisfies $\pi(y)=x$ \cite[Chapt. V, \S 2.2, Proposition 6(ii)]{bourbaki-inglese}. Hence, if $\overline{\theta}(z)$ and $\overline{\theta}(w)$ are conjugates there is a $\overline{\rho}\in\Gal(\overline{k}/k)$ such that $\overline{\rho}(\overline{\theta}(z))=\overline{\theta}(w)$, and we have
\begin{align*}
\overline{\rho}\left(\overline{\pi}\left(\frac{z-\alpha}{c}\right)\right)=\overline{\pi}\left(\frac{w-\alpha}{c}\right) & \iff\overline{\pi}\left(\rho\left(\frac{z-\alpha}{c}\right)\right)=\overline{\pi}\left(\frac{w-\alpha}{c}\right)\\ &\iff\rho\left(\frac{z-\alpha}{c}\right)-\frac{w-\alpha}{c}\in M_U.
\end{align*}
Since $\alpha,c\in K$, the last condition holds if and only if $\rho(z)-w\in cM_U$. Conversely, if $\rho(z)-w\in cM_U$, then we can follow the same reasoning in the opposite order, and so $\overline{\theta}(z)$ and $\overline{\theta}(w)$ are conjugate over $k$.
\end{proof}

We conclude by reproving Ostrowski's Fundamentalsatz. Recall that, if $V$ has rank $1$, we can always consider $v$ as a (not necessarily surjective) map from $K\setminus\{0\}$ to $\insR$.

\begin{teor}\label{Ostrowski fundamentalsatz}
Suppose that $v$ is a valuation of rank $1$ and $K$ is algebraically closed. Let $w$ be an extension of $v$ to $K(X)$ of rank $1$. Then the following hold:
\begin{enumerate}[(a)]
\item there is a pseudo-convergent sequence $E=\{s_\nu\}_{\nu\in\Lambda}\subset K$ such that
\begin{equation*}
w(\phi)=\lim_{\nu\to\infty}v(\phi(s_\nu))
\end{equation*}
for every nonzero $\phi\in K(X)$;
\item if $V\subset W$ is not immediate, there is also a pseudo-divergent sequence $F=\{t_\nu\}_{\nu\in\Lambda}$ such that
\begin{equation*}
w(\phi)=\lim_{\nu\to\infty}v(\phi(t_\nu))
\end{equation*}
for every nonzero $\phi\in K(X)$.
\end{enumerate}
\end{teor}
\begin{proof}
If $V\subset W$ is immediate, then by \cite[Theorems 1 and 3]{kaplansky-maxfield} $W=V_E$ for some pseudo-convergent sequence $E$ of transcendental type and $w(\phi)=v_E(\phi)=v(\phi(s_\nu))$ for $\nu\geq N(\phi)$.

Suppose now that $V\subset W$ is not immediate. By Theorem \ref{teor:fundamentalsatz}, there is a pseudo-monotone sequence $G\subset K$ such that $W=V_G$ with $\inslim_G\neq\emptyset$. We distinguish two cases.

Suppose first that $G$ is pseudo-stationary. Then, $\Br(G)=cV$, and $v_G(\phi)=\lambda_\phi\delta+\gamma$, where $\lambda_\phi=\degdom_G(\phi)$, $\delta=v(c)$ and $\gamma=v\left(\frac{\phi}{\phi_S}(\beta)\right)$ for some pseudo-limit $\beta$ of $G$ in $K$. Let $E=\{s_\nu\}_{\nu\in\Lambda}\subset K$ be a pseudo-convergent sequence such that $\inslim_E=\beta+cV=\inslim_G$ (Lemma \ref{lemma:carattbreadth}); then, $\degdom_E(\phi)=\degdom_G(\phi)$ and the gauge $\{\delta_\nu\}_{\nu\in\Lambda}$ of $E$ tends to $\delta$, the gauge of $G$. By Proposition \ref{values rat func on psm seq},
\begin{equation*}
\lim_{\nu\to\infty}v(\phi(s_\nu))=\lim_{\nu\to\infty}(\lambda_\phi\delta_\nu+\gamma)=\lambda_\phi\left(\lim_{\nu\to\infty}\delta_\nu\right)+\gamma=\lambda_\phi\delta+\gamma=v_G(\phi)=w(\phi),
\end{equation*}
and the claim is proved. In the same way, we can find a pseudo-divergent sequence $F=\{t_\nu\}_{\nu\in\Lambda}\subset K$ such that $\inslim_F=\beta+cM$; as in the proof of Proposition \ref{prop:subext-VF}, setting $\{\delta'_\nu\}_{\nu\in\Lambda}$ to be the gauge of $F$, we have (for large $\nu$)
\begin{equation*}
v(\phi(t_\nu))=\lambda'\delta'_\nu+(\lambda_\phi-\lambda')\delta+\gamma,
\end{equation*}
where $\lambda'=\degdom_F(\phi)$. Hence, $v(\phi(t_\nu))\to \lambda_\phi\delta+\gamma=w(\phi)$, as claimed.

Suppose now that $G$ is strictly pseudo-monotone, and let $\beta\in\inslim_G$. If $\Br(G)$ is equal to $cV$ or to $cM$ for some $c\in K$, then we can find a pseudo-stationary sequence $G'$ with breadth ideal $cV$ and having $\beta$ as a pseudo-limit; by the discussion at the beginning of the section and by Proposition \ref{prop:subext-VF}, $V_{G'}$ would properly contain $V_G$, against the fact that $V_G$ has rank one. Therefore, $\Br(G)$ is both strictly divisorial and nonprincipal; by Lemma \ref{lemma:carattbreadth} we can find a pseudo-convergent sequence $E$ and a pseudo-divergent sequence $F$ in $K$ such that $\inslim_E=\inslim_F=\inslim_G=\beta+\Br(G)$ (note that one between $E$ and $F$ could be taken equal to $G$). In particular, $\Br(E)=\Br(F)=\Br(G)$ and so $\delta_E=\delta_F=\delta$. 

Since $W=V_E$ has rank 1, by \cite[Theorem 4.9(c)]{PS} the valuation relative to $V_E$ is exactly the one mapping $\phi\in K(X)$ to
\begin{equation*}
\lambda\delta+\gamma=\lim_{\nu\to\infty}(\lambda\delta_\nu+\gamma)
\end{equation*}
where $\lambda=\degdom_E(\phi)=\degdom_F(\phi)$ and $\gamma=v\left(\frac{\phi}{\phi_S}(\beta)\right)$. Since $\delta$ is also the limit of $\delta'_\nu$, the claim is proved.
\end{proof}

\end{document}